\documentclass[10pt]{article}

\usepackage[left=1in,top=1in,right=1in,bottom=1in,letterpaper]{geometry}

\usepackage[colorlinks]{hyperref}
\usepackage{url}
\usepackage{amsthm,amsmath,amssymb, amscd}
\usepackage{mathrsfs}
\usepackage{amsfonts}
\usepackage[utf8]{inputenc}
\usepackage{subfigure}
\usepackage{graphicx}
\usepackage{bbding}
\everymath{\displaystyle}
\usepackage{indentfirst} 
\usepackage{enumerate}
\usepackage{bm}
\usepackage{enumitem}
\usepackage{algorithm, algorithmic}
\usepackage[title]{appendix}
\usepackage{dsfont}
\usepackage{booktabs}
\usepackage{multirow}
\usepackage{threeparttable}
\usepackage{makecell}
\usepackage{colortbl}
\usepackage{comment}

\definecolor{Ocean}{RGB}{129,194,234}
\definecolor{BPink}{rgb}{0.96, 0.76, 0.76}
\usepackage[font={small,it}]{caption}
\usepackage{xspace}

\newtheorem{thm}{Theorem}

\newtheorem{assumption}{Assumption}
\newtheorem{lemma}[thm]{Lemma}

\newtheorem{proposition}[thm]{Proposition}

\newtheorem{definition}{Definition}
\newtheorem{theorem}{Theorem}
\usepackage{apptools}

\date{}

\usepackage{booktabs} 
\usepackage{makecell}
\usepackage{mathtools}
\usepackage{xspace}

\usepackage{amsmath}
\usepackage{enumitem}

\newcommand{\vf}{{\mathbf{f}}}
\newcommand{\vg}{{\mathbf{g}}}

\newcommand{\vw}{{\mathbf{w}}}
\newcommand{\vx}{{\mathbf{x}}}
\newcommand{\vy}{{\mathbf{y}}}
\newcommand{\vz}{{\mathbf{z}}}

\newcommand{\cA}{{\mathcal{A}}}

\newcommand{\cE}{{\mathcal{E}}}
\newcommand{\cF}{{\mathcal{F}}}
\newcommand{\cG}{{\mathcal{G}}}

\newcommand{\cO}{{\mathcal{O}}}

\newcommand{\EE}{\mathbb{E}}
\newcommand{\RR}{\mathbb{R}}

\newcommand{\bxi}{\boldsymbol{\xi}}

\newcommand{\one}{\mathds{1}}

\newcommand{\eg}{{\em e.g.\xspace}}
\newcommand{\ie}{{\em i.e.\xspace}}

\newcommand{\prog}{\mathrm{prog}}

\newcommand{\pushsum}{{\sc Push-Sum}\xspace}
\newcommand{\pushdg}{{\sc Push-DIGing}\xspace}
\newcommand{\mgpushdg}{{\sc MG-Push-DIGing}\xspace}
\newcommand{\sgd}{{\sc SGD}\xspace}
\newcommand{\ed}{{\sc Exact-Diffusion}\xspace}
\newcommand{\ds}{{\sc D}$^2$\xspace}
\newcommand{\extrapush}{{\sc EXTRA-Push}\xspace}
\newcommand{\dextra}{{\sc DEXTRA}\xspace}
\newcommand{\addopt}{{\sc ADD-OPT}\xspace}
\newcommand{\saddopt}{{\sc S-ADDOPT}\xspace}
\newcommand{\ab}{{\sc AB}\xspace}
\newcommand{\pushpull}{{\sc Push-Pull}\xspace}
\newcommand{\subpush}{{\sc Subgradient-Push}\xspace}
\newcommand{\sgdpush}{{\sc SGD-Push}\xspace}
\newcommand{\norm}[1]{\| #1 \|}

\newcommand{\pinorm}[1]{\| #1 \|_\pi}
\newcommand{\fnorm}[1]{\| #1 \|_F}
\begin{document}

\title{Understanding the Influence of Digraphs on Decentralized Optimization: Effective Metrics, Lower Bound, and Optimal Algorithm}
\author{Liyuan Liang\thanks{Equal contribution. Liyuan Liang is with School of Mathematics Science, Peking University ({2000010643@stu.pku.edu.cn}), and Xinmeng Huang is with Graduate Group in Applied Mathematics and Computational Science, University of Pennsylvania ({xinmengh@sas.upenn.edu}).}
\and Xinmeng Huang$^*$\hspace{-1.5mm}
  \and Ran Xin\thanks{ByteDance Applied Machine Learning ({ran.xin@bytedance.com})}
\and Kun Yuan\thanks{Corresponding author. Kun Yuan is with Center for Machine Learning Research, Peking University ({kunyuan@pku.edu.cn})}}

\maketitle

\begin{abstract}
This paper investigates the influence of directed networks on decentralized stochastic non-convex optimization associated with column-stochastic mixing matrices. Surprisingly, we find that the canonical spectral gap, a widely used metric in undirected networks, is insufficient to characterize the impact of directed topology on decentralized algorithms. To overcome this limitation, we introduce a novel metric termed \emph{equilibrium skewness}. This metric, together with the spectral gap, accurately and comprehensively captures the influence of column-stochastic mixing matrices on decentralized stochastic algorithms. With these two metrics, we clarify, for the first time, how the directed network topology influences the performance of prevalent algorithms such as \pushsum and \pushdg. Furthermore, we establish the first lower bound of the convergence rate for decentralized stochastic non-convex algorithms over directed networks. Since existing algorithms cannot match our lower bound, we further propose the \mgpushdg algorithm, which integrates \pushdg with a multi-round gossip technique. \mgpushdg attains our lower bound up to logarithmic factors, demonstrating its near-optimal performance and the tightness of the lower bound. Numerical experiments verify our theoretical results.

\end{abstract}

\vspace{-4pt}
\section{Introduction}

This paper considers decentralized stochastic optimization defined over a network of $n$ nodes:
\vspace{-2pt}
\begin{align}\label{prob-general}
\min_{x\in \mathbb{R}^d}\quad f(x) := \frac{1}{n}\sum_{i=1}^n f_i(x) \quad \mbox{where} \quad f_i(x) = \mathbb{E}_{\xi_i \sim \mathcal{D}_i}[F(x; \xi_i)].
\end{align}
Here, $\xi_i$ is a random data vector supported on $\Xi_i\subseteq\mathbb{R}^q$ with some distribution~$\mathcal{D}_i$ and~$F:\mathbb{R}^d\times\mathbb{R}^q\rightarrow\mathbb{R}$ is a Borel measurable function. Each loss function~$f_i$ is only accessible locally by node~$i$ and is assumed to be smooth and possibly non-convex. Note that data heterogeneity typically exists, \ie, local data distributions $\{\mathcal{D}_i\}_{i=1}^n$ differ across the nodes. The decentralized communication among the networked nodes is abstracted as a strongly-connected \emph{directed graph}. The family of directed communication networks is of practical significance. For instance, bi-directional communication may be impossible due to nodes having varying power range~\cite{yang2019survey}.
In the context of distributed training of deep learning models, well-designed directed topologies often result in sparser and faster communication compared to their undirected counterparts, and thus accelerate training in terms of 
wall-clock time; see, \eg, \cite{bottou2018optimization,assran2019stochastic,yuan2021decentlam}. 


\subsection{Influence of undirected network is well studied}
The fundamental difference between decentralized and centralized optimization lies in the communication network they utilize: the former employs arbitrary decentralized communication topologies, while the latter operates on a star network. Hence, a primary challenge in decentralized algorithms is the characterization of the influence of network topologies on their convergence properties and practical performance. This challenge has been well addressed for~algorithms over {\em undirected} networks, where we associate the network of interest with a {\em doubly-stochastic} mixing matrix $W \in\mathbb{R}^{n\times n}$ that follows its sparsity pattern. The {\em spectral gap} of  $W$ is defined as $1-\beta$ where 
\begin{align}\label{eqn:beta-undirected}
\beta := \left\|W - \mathds{1}_n\mathds{1}_n^\top/n\right\|_2 \in [0,1) \quad \mbox{and} \quad \mathds{1}_n:=(1,\dots,1)^\top \in \mathbb{R}^n.
\end{align}
Spectral gap is a widely recognized metric in the literature to evaluate the connectivity of an undirected network. A value of $1-\beta$ approaching $1$ indicates a well-connected network, whereas a value close to $0$ suggests poor connectivity. Specifically, the convergence rate of a decentralized optimization algorithm over an undirected network improves as the spectral gap increases; see the left plot of Figure \ref{fig:intro} for an example. For theoretical results, we refer the readers to \cite{koloskova2020unified,yuan2023removing,alghunaim2022unified,GT-Xinran,song2023optimal} and the references therein.

\begin{figure}[t]
    \centering
    \includegraphics[width=2.2in]{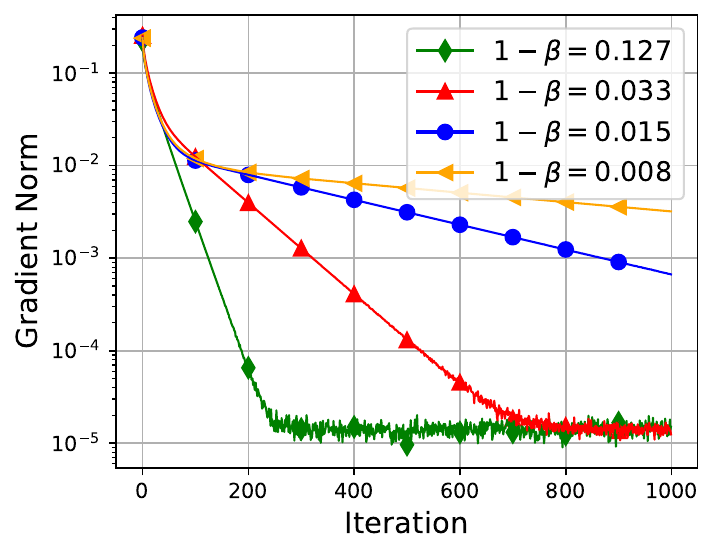}
    \quad \quad 
    \includegraphics[width=2.2in]{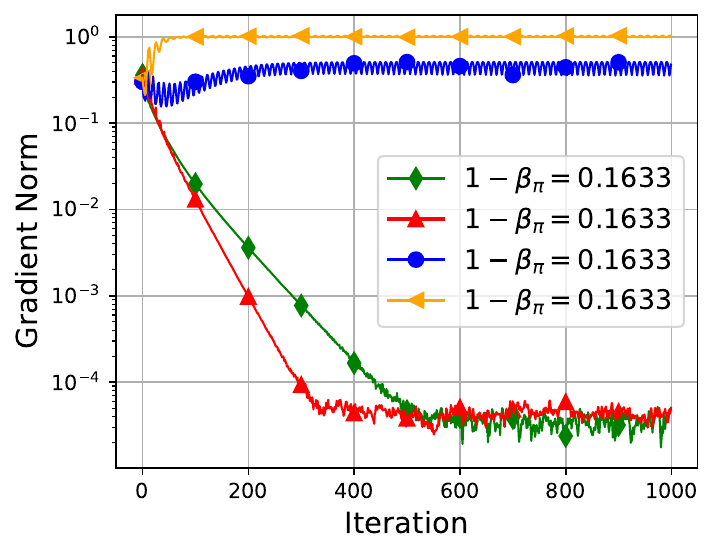}
    \vspace{-2mm}
    \caption{Convergence performance of decentralized gradient tracking~\cite{koloskova2020unified,yuan2023removing,alghunaim2022unified,GT-Xinran} for the regularized logistic regression problem. Left plot: The spectral gap accurately reflects the influence of {\em undirected} graphs on gradient tracking. Right plot: The spectral gap fails to reflect the influence of {\em directed} graphs on gradient tracking. Experimental details are deferred to Sec.~\ref{sec:experiments}. }
    \label{fig:intro}
    \vspace{-0.4cm}
\end{figure}

\subsection{Influence of directed network remains unclear} \label{sec:intro-digraph}
Directed networks, also commonly known as digraphs, represent a more general class of topologies, encompassing undirected networks as a special case. The central challenge of optimization over directed networks is
its general inability to construct a doubly-stochastic mixing matrix~\cite{tsianos2012push}. In other words, the mixing matrices typically exhibit either column-stochasticity or row-stochasticity, but not both. Most of the existing literature associates a digraph with a column-stochastic mixing matrix whose construction can be readily implemented in a decentralized manner using the out-degree of each node. For instance, references~\cite{Olshevsky2015Dist,tsianos2012push,zeng2017extrapush,xi2017dextra,xi2017add,Nedic2017pushdiging,assran2019stochastic,qureshi2020s} develop decentralized algorithms to solve problem \eqref{prob-general} over digraphs with \pushsum, a technique~\cite{kempe2003gossip,tsianos2012push} that uses column-stochastic mixing matrices to achieve consensus among the nodes. Similar to doubly-stochastic mixing matrics in undirected networks, we may define 
the {\em generalized spectral gap} $1-\beta_\pi$ of a column-stochastic $W$~\cite{xin2019SAB}, where
\begin{align}\label{eqn:beta-directed}
\beta_\pi := \left\|W - \pi \mathds{1}_n^\top\right\|_\pi \in [0,1).
\end{align}
Here, $\pi \in \RR^n$ is the normalized right eigenvector of $W$ corresponding to~the eigenvalue~$1$ such that $W \pi = \pi$ and $\mathds{1}_n^\top \pi = 1$. Such $\pi$ is also called the {\em equilibrium vector} and $\|\cdot\|_\pi$ is a weighted matrix norm, formally defined in Sec.~\ref{sec:notation}. Notably, for doubly-stochastic matrices, $\pi=\one_n/n$ and thus $\beta_\pi$ coincides with $\beta$ defined in \eqref{eqn:beta-undirected}.

Surprisingly, we find that the generalized spectral gap metric~\eqref{eqn:beta-directed} does not adequately capture the impact of digraphs on decentralized optimization. Specifically, we construct a family of column-stochastic matrices that share the same directed topology and generalized spectral gap. Despite these shared characteristics, decentralized optimization algorithms exhibit significantly different convergence rates with these weight matrices. This observation is illustrated in the right plot of Figure \ref{fig:intro}, which highlights that understanding the influence of directed networks on decentralized optimization is not a straightforward extension of the undirected network literature and makes imperative the development of new digraph metrics and analytical methodologies. The following questions then naturally arise: 
\begin{itemize}[leftmargin=2em]
\vspace{1.5mm}
\item {What are the effective metrics that can fully capture the impact of digraphs on decentralized stochastic optimization? How do they influence the convergence of prevalent algorithms such as \pushsum averaging and \pushdg?}

\vspace{1.5mm}
\item {Given these metrics, what is the lower bound of convergence rate for decentralized stochastic non-convex optimization with column-stochastic mixing matrices?}

\vspace{1.5mm}
\item {Can we develop an optimal algorithm to attain the above lower bound?}
\end{itemize}

\begin{figure}[t]
\centering
\vspace{-5mm}
\includegraphics[width=6in]{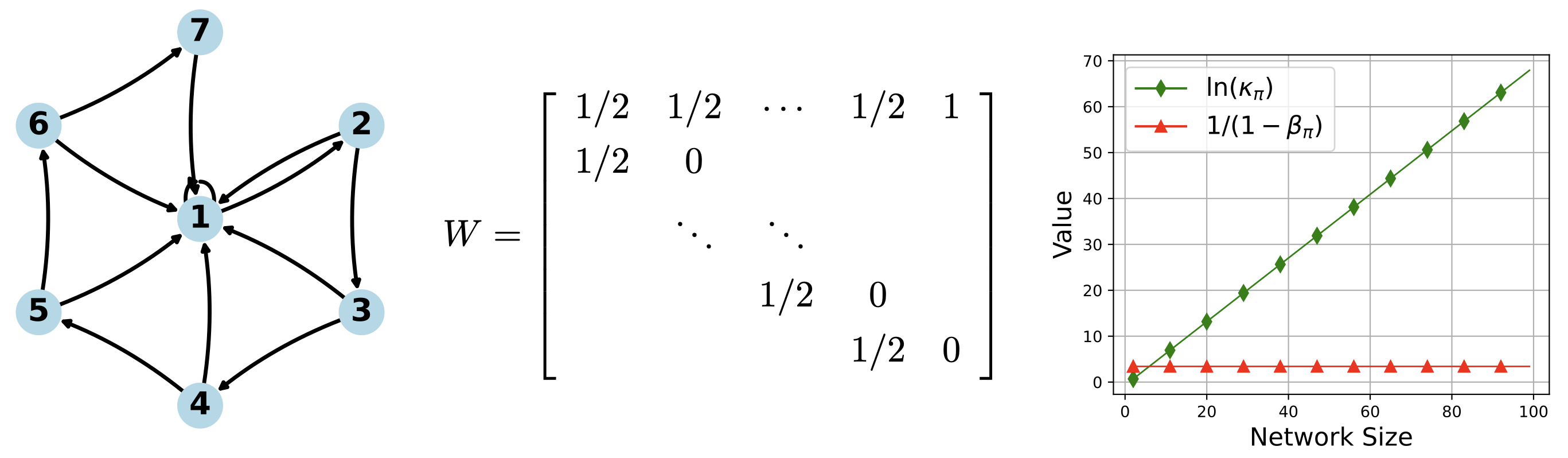}
\vspace{-2mm}
\caption{Left plot: A delicately-designed directed network topology with $n=7$. Middle plot: A corresponding column-stochastic matrix $W$. Right plot: The equilibrium skewness $\kappa_\pi$ increases exponentially fast with size $n$ while the spectral gap remains constant.}\label{fig:intro-2}
\vspace{-0.4cm}
\end{figure}
\subsection{Main results}
This paper elucidates the impact of digraphs on decentralized optimization, by providing affirmative answers to these fundamental questions.

\subsubsection{Equilibrium skewness and its orthogonality to the generalized spectral gap} We introduce a novel metric called
{\em equilibrium skewness},
\vspace{-2pt}
\begin{align}\label{eqn:kappa_pi}
\kappa_\pi := \max_i \pi_i/\min_i \pi_i \in [1,+\infty),
\end{align}
where $\pi$, defined in \eqref{eqn:beta-directed}, is the equilibrium vector of an associated column-stochastic mixing matrix and~$\kappa_\pi$ effectively captures the disagreement magnitude between the maximum and minimum values therein. {We will demonstrate that the generalized spectral gap and the equilibrium skewness, in conjunction, can accurately capture the impact of column-stochastic mixing matrices on decentralized stochastic algorithms, as discussed in Sec.~\ref{intro-lowerBound} and \ref{intro-pushdg}}.
Additionally, we further prove that $\kappa_\pi$ and the spectral gap~$1-\beta_\pi$ can be \emph{orthogonal} to each other over certain mixing matrices. 
\begin{proposition}[Informal]\label{prop:intro-special-mat}
  For any directed graph with size $n\geq1$, there exists an associated column-stochastic matrix $W\in\RR^{n\times n}$ such that $\beta_\pi=\frac{\sqrt{2}}{2}$ but $\kappa_\pi=2^{n-1}$.
\end{proposition}

This proposition indicates that $\kappa_\pi$ can grow exponentially fast as the network size $n$ increases, while $1-\beta_\pi$ remains a constant at the same time. The experiments shown in Figure~\ref{fig:intro-2} give a concrete example for Proposition \ref{prop:intro-special-mat}.



\begin{table}[!t]
    \caption{Upper and Lower bounds for non-convex decentralized stochastic optimization with column-stochastic mixing matrices. Quantities $K,\sigma^2,L,n,\Delta$ denote the number of iterations, gradient variance, the smoothness parameter of functions, the network size, the initial function value gap respectively. ``Rate (A.)'' represents the asymptotic convergence rate as $K\to \infty$. ``Rate (F.T.)'' represents the non-asymptotic finite-time convergence rate which reflects the impact of digraphs on the performance. ``N.A.'' means not available in the corresponding reference.}
    \centering
    \begin{tabular}{l c c c}
        Algorithm & Rate (A.)  & Rate (F.T.) &Transient Time \\
        \toprule
        Gradient-Push~\cite{assran2019stochastic}&$\frac{ \sigma \sqrt{L\Delta}}{\sqrt{nK}
        }$&N.A.& N.A. \vspace{0.1cm}\\
        Push-DIGing~\cite{kungurtsev2023decentralized}&$\frac{ \sigma \sqrt{L\Delta}}{\sqrt{nK}
        }$&N.A.&N.A.\vspace{0.1cm}\\
        Push-DIGing~\hspace{-0.7mm}({\bf Ours})&  $\frac{ \sigma \sqrt{L\Delta}}{\sqrt{nK}
        }$&  $\hspace{-3mm} \frac{ \sigma \sqrt{L\Delta}}{\sqrt{nK}
        }+\sqrt[3]{\frac{\kappa_\pi^6\beta_\pi^4L^2\Delta^2\sigma^2}{(1-\beta_\pi)^3K^2}}^\dag$ \vspace{0.1cm}\hspace{-6mm} &$\frac{n^3\kappa_\pi^{14}}{(1-\beta_\pi)^6}$ \\
        MG-Push-DIGing~\hspace{-.7mm}({\bf Ours}) &  $\frac{ \sigma \sqrt{L\Delta}}{\sqrt{nK}
        }$&  $\frac{ \sigma \sqrt{L\Delta}}{\sqrt{nK}
        }+\frac{(1+\ln(\kappa_\pi))L\Delta}{(1-\beta_\pi)K}$ \vspace{0.05cm}&$\frac{n(1+\ln(\kappa_\pi))^2}{(1-\beta_\pi)^2}$\\
        \midrule
        Lower Bound~\hspace{-.7mm}({\bf Ours})&  $\frac{ \sigma \sqrt{L\Delta}}{\sqrt{nK}
        }$ & $\frac{ \sigma \sqrt{L\Delta}}{\sqrt{nK}
        }+\frac{(1+\ln(\kappa_\pi))L\Delta}{(1-\beta_\pi)K}$  \vspace{0.05cm}&$\frac{n(1+\ln(\kappa_\pi))^2}{(1-\beta_\pi)^2}$\\
        
        \bottomrule
    \end{tabular}
    \begin{tablenotes}
    \footnotesize
    \item $^\dag$We only list the terms that determine the order of transient time for clarity. We refer the full rate to Theorem~\ref{thm:push diging convergence}.
    \end{tablenotes}
    \label{tab:compare}
    \vspace{-0.3cm}
\end{table}

\subsubsection{Lower bound} \label{intro-lowerBound}
With $\beta_\pi$ and $\kappa_\pi$ at hand, 
we show, {\em for the first time}, that the convergence rate of any non-convex decentralized stochastic first-order algorithm with a column-stochastic mixing matrix is lower bounded by
\begin{align}\label{eqn:intro-lower-bound-intro}
\Omega\left(\frac{ \sigma \sqrt{L\Delta}}{\sqrt{nK}
} + \frac{(1+\ln(\kappa_\pi))L\Delta}{(1-\beta_\pi)K}\right),
\end{align}
where $K,\sigma, L,\Delta$ denote the total number of iterations, gradient variance, smoothness parameter of functions, and initial function value gap, respectively. This lower bound~\eqref{eqn:intro-lower-bound-intro} explicitly reveals the joint impact of the generalized spectral gap $\beta_\pi$ and the equilibrium skewness $\kappa_\pi$
on the non-asymptotic convergence rate of decentralized optimization algorithms and renders both metrics essential in the corresponding characterization and convergence analysis, irrespective of algorithmic designs. It is worth noting that, when the gradient noise is absent, \ie, $\sigma = 0$, \eqref{eqn:intro-lower-bound-intro} reduces to the lower bound for the class of decentralized \emph{deterministic} first-order algorithms with column-stochastic mixing matrices, which also appears to be new in the literature.

\subsubsection{\pushdg revisited and optimal algorithm} 
\label{intro-pushdg}
We establish refined convergence rates for \pushsum and \pushdg, which fully characterize the impact of general directed topologies, \ie, $\beta_\pi$ and $\kappa_\pi$, on the performance of these algorithms. A comparison between our findings and the existing results is presented in Table \ref{tab:compare}. {It can be observed that existing literature~\cite{assran2019stochastic,kungurtsev2023decentralized} has not elucidated the impact of digraphs on the performance of decentralized algorithms, and the vanilla \pushdg algorithm fails to attain the lower bound defined in \eqref{eqn:intro-lower-bound-intro}.} To resolve this issue, we further propose \mgpushdg, \ie, \pushdg with multiple gossip communications, to achieve a rate on the order of
\vspace{-2pt}
\begin{align}\label{eqn:intro-upper-bound-intro}
\widetilde{\cO}\left(\frac{ \sigma \sqrt{L\Delta}}{\sqrt{nK}
} + \frac{(1+\ln(\kappa_\pi))L\Delta}{(1-\beta_\pi)K}\right).
\end{align}
Here, we absorb logarithmic factors that are independent of $\kappa_\pi$ and $\beta_\pi$ in $\widetilde{\cO}(\cdot)$. Therefore, MG-Push-DIGing nearly attains the lower bound \eqref{eqn:intro-lower-bound-intro}, demonstrating its near-optimal performance and the tightness of the established lower bound. To the best of our knowledge, MG-Push-DIGing is the first rate-optimal algorithm in the literature of decentralized stochastic non-convex optimization over general directed graphs.

\subsection{Related work} 
Utilizing the concept of the spectral gap, a large collection of results have been established to characterize the impact of undirected networks on decentralized algorithms. References such as \cite{sayed2014adaptive, koloskova2020unified, pu2021sharp, ying2021exponential, chen2021accelerating} demonstrate that decentralized stochastic gradient descent (D\sgd) achieves the same convergence rate as parallel \sgd, which operates on a central server, after a finite number of transient iterations, specifically $\cO(n^3/(1-\beta)^4)$. In large and poorly-connected networks where $1-\beta \to 0$, this transient stage is notably prolonged, potentially making DSGD lag behind parallel \sgd throughout the entire algorithm execution timeframe. To mitigate this issue, references \cite{alghunaim2022unified, yuan2023removing,huang2022improving,tang2018d} propose \ed/\ds and  gradient tracking to shorten the transient stage by eliminating the influence of data heterogeneity,
while \cite{chen2021accelerating} leverages periodic global averaging to compensate for decentralization. 

For general directed networks, there also exist various decentralized algorithms to solve problem \eqref{prob-general}. When the exact gradients $\nabla f_i$'s are available, the well-known \subpush  algorithm \cite{Olshevsky2015Dist,tsianos2012push} converges to the desired solutions, but with relatively slow sublinear convergence even under strong convexity. Advanced methods such as \extrapush \cite{zeng2017extrapush}, \dextra \cite{xi2017dextra}, \addopt \cite{xi2017add}, and \pushdg~\cite{Nedic2017pushdiging} achieve faster convergence by alleviating the impact of data heterogeneity. For stochastic gradients, \sgdpush \cite{assran2019stochastic} and \saddopt \cite{qureshi2020s} asymptotically achieve the same rate as parallel \sgd. The concept of the generalized spectral gap~$\beta_\pi$ in directed networks was initially introduced by \cite{xin2019SAB} and subsequently applied in \cite{xin2018linear,pu2020push}. Yet, these works do not explicitly show the dependence of $\beta_\pi$ and ignore the effect of equilibrium skewness $\kappa_\pi$ in convergence. To our knowledge, no research has thoroughly examined the joint impact of the spectral gap and equilibrium skewness on decentralized algorithms with column-stochastic mixing matrices.

While this paper studies decentralized optimization over column-stochastic mixing matrices, some other studies such as \cite{sayed2014adaptive, xin2019frost, xi2018linear, yuan2018exact} focus on row-stochastic ones. It is shown in \cite{xin2018linear,pu2020push} that optimization algorithms relying solely on a column-stochastic or row-stochastic mixing matrix converge relatively slowly. To significantly accelerate convergence, a novel \ab algorithm \cite{xin2018linear} (also known as \pushpull \cite{pu2020push}) is proposed, which alternates between column- and row-stochastic mixing matrices in each update. However, the influence of directed network topologies still remains unclear in all these existing works. We conjecture that the equilibrium skewness plays a similar role in these settings but leave the systematic study to future work.

\subsection{Notations}\label{sec:notation}
Throughout the paper, we let $x^{(k)}_{i}\in \mathbb{R}^d$ denote the local model copy at node $i$ at iteration $k$. Furthermore, we define the matrices
\vspace{-2pt}
\begin{align*}
    \vx^{(k)} &:= [(x_1^{(k)})^\top; (x_2^{(k)})^\top; \cdots; (x_n^{(k)})^\top]\in \mathbb{R}^{n\times d}, \nonumber\\
    \nabla F(\vx^{(k)};\bxi^{(k)}) &:= [\nabla F_1(x_1^{(k)};\xi_1^{(k)})^\top; \cdots; \nabla F_n(x_n^{(k)};\xi_n^{(k)})^\top]\in \mathbb{R}^{n\times d},\nonumber\\
    \nabla f(\vx^{(k)}) &:= [\nabla f_1(x_1^{(k)})^\top; \nabla f_2(x_2^{(k)})^\top; \cdots; \nabla f_n(x_n^{(k)})^\top]\in \mathbb{R}^{n\times d}, \nonumber
\end{align*}
by stacking all local variables. We similarly denote the average of local parameters $n^{-1}\sum_{i=1}^nx_i^{(k)}\in\RR^d$ as $\bar x^{(k)}$ and denote  $[(\bar x^{(k)})^\top; (\bar x^{(k)})^\top; \cdots; (\bar x^{(k)})^\top]\in \mathbb{R}^{n\times d}$ as $\bar \vx^{(k)}$.
Here, upright bold symbols (\eg, $\vx,\vf,\vw$) denote augmented network-level quantities. We use $\mathrm{col}\{a_1,\dots,a_n\}$ and $\mathrm{diag}\{a_1,\dots,a_n\}$ to denote a column vector and a diagonal matrix formed from scalars $a_1,\dots,a_n$. We let $\mathds{1}_n:=\mathrm{col}\{1,\dots,1\}\in \mathbb{R}^n$ and $I_n\in \mathbb{R}^{n\times n}$ denote the identity matrix. Given two matrices $\vx, \vy \in \mathbb{R}^{n\times d}$, we define inner product $\langle \vx, \vy \rangle := \mathrm{tr}(\vx^\top \vy)$. For a matrix $A$, we let $\|A\|$ denote its ${l}_2$ norm and $\|A\|_F$ denote its Frobenius norm and refer its $j$-th column to $A_{\cdot,j}$. We use $\|\cdot\|_p$ to denote the $l_p$ vector norm for $p\geq 1$. 
 For a vector $\pi\in\mathbb{R}^n$ with positive entries, we represent the entry-wise $\alpha$-th power by $\pi^{\alpha}$, where $\alpha > 0$. The vector $\pi$-norm \cite{xin2019SAB} is defined as $\|v\|_\pi:=\|\mathrm{diag}(\sqrt{\pi})^{-1/2}v\|_2$ for $v\in\RR^n$. We define the induced matrix $\pi$-norm as $\|A\|_\pi:=\|\mathrm{diag}(\sqrt{\pi})^{-1}A \ \mathrm{diag}(\sqrt{\pi})\|_2$ for $A\in\RR^{n\times n}$. We let $\underline{\pi}:=\min_i\pi_i$, $\overline{\pi}:=\max_i\pi_i$. We let $\Pi:=\mathrm{diag}(\pi)$. We use  $\gtrsim$ and $\lesssim$ to denote inequalities that hold up to universal constants such as $L$ and $\sigma$.

\section{Finding the effective metrics}\label{sec:prelim}
We consider a directed network with $n$ computing nodes that is associated with a mixing matrix $W=[w_{ij}]_{i,j=1}^n \in\RR^{n\times n}$ where $w_{ij} \in (0,1)$ if node $j$ sends information to node $i$ otherwise $w_{ij} = 0$. Decentralized optimization is built upon partial averaging $z_i^+ = \sum_{j\in \mathcal{N}_i}w_{ij} z_j$ in which $\mathcal{N}_i$ denotes the in-neighbors of node $i$, including node $i$ itself. Since every node conducts partial averaging simultaneously, we have 
\vspace{-2pt}
\begin{equation}\label{eqn:w-protocol}
\vz \triangleq [z_1^\top; z_2^\top; \cdots; z_n^\top]\ \xmapsto{\text{W\text{-protocol}}} \ \vz^+ = W\vz =\left[\sum_{j=1}^nw_{1j}z_j^\top;\cdots;\sum_{j=1}^nw_{nj}z_j^\top\right]
\end{equation}
where $W$-protocol represents partial averaging with mixing matrix $W$. Evidently, the algebraic characteristics of $W$ substantially affects the convergence of partial averaging and the corresponding decentralized optimization. In this section, we find these characteristics for general digraphs.

\subsection{Column stochastic mixing matrix} Throughout this paper, we focus on a static directed network $\cG$ associated with a column-stochastic matrix $W$.
\begin{assumption}[\sc Primitive and Column-stochastic mixing matrix] \label{ass-weight-matrix}
The mixing matrix $W$ is non-negative, primitive, and satisfies $\mathds{1}_n^\top W = \mathds{1}_n^\top$. 
\end{assumption}

If $\cG$ is strongly-connected, \ie, there exists a directed path from each node to every other node, and $W$ has positive trace, then $W$ is primitive. 
It is also straightforward to make $W$ column-stochastic, by setting weights as 
\vspace{-2pt}
\begin{equation*}
w_{ij}=\begin{cases}
    1/(1+d_j^{\rm out}) & \text{if }(j, i)\in \cE \text{ or }j=i,\\
    0 &\text{otherwise},
\end{cases}
\end{equation*}
where $\cE$ is the set of directed edges and $d_i^{\rm out}$ is the out-degree of node $i$ excluding the self-loop.
Under Assumption~\ref{ass-weight-matrix}, the Perron-Frobenius theorem~\cite{Perron1907} guarantees a unique equilibrium vector $\pi\in\mathbb{R}^n$ with positive entries such that
\vspace{-2pt}
\begin{align*}
W\mathds\pi = \mathds\pi \quad \mbox{and}\quad \mathds{1}_n^\top \pi=1.
\end{align*}

\subsection{Push-Sum strategy} 
The partial averaging, $\vz^{(k+1)}=W\vz^{(k)}$ for $k\geq 0$ with $\vz^{(0)}=\vz$, eventually drives $\vz^{(k)}$ to the fixed point $\pi \one_n^\top \vz$ due to the fact that $\lim_{k\to \infty}W^k = \pi \one_n^\top$. This implies that $\vz^{(k)}$ does not converge to the global average $n^{-1}\one_n \one_n^\top \vz$. Such bias can be fixed by adding a correction mechanism to the update. Note that
the average $n^{-1}\one_n \mathds{1}_n^\top \vz$ may be achieved by reweighting with $\mathrm{diag}(n\pi)^{-1}$:
\vspace{-2pt}
\begin{equation*}
    \mathrm{diag}(n\pi)^{-1} \vz^{(k)}\ \longrightarrow\ \mathrm{diag}(n\pi)^{-1}\pi \one_n^\top \vz=n^{-1}\one_n\one_n^\top \vz=:\bar\vz.
\end{equation*}
Since $\pi$ is not known by the nodes a priori, one may perform decentralized power iterations $v^{(k+1)}=Wv^{(k)}$ with $\mathds{1}_n^\top v^{(0)}=n$, to obtain $v^{(k)}\to n \pi$, \ie, each node learns $\pi$ asymptotically. The above discussion leads to the \pushsum strategy \cite{kempe2003gossip,tsianos2012push}:
\vspace{-2pt}
\begin{equation}\label{eqn:push-sum}
\begin{split}
       \vz^{(k+1)}&=W\vz^{(k)} \\
    v^{(k+1)}&=W v^{(k)} \\
    V^{(k+1)}    &=\mathrm{diag}(v^{(k+1)}) \\
    \vw^{(k+1)}&={V^{(k+1)}}^{-1}\vz^{(k+1)}. 
\end{split}
\end{equation}
\pushsum achieves the global average asymptotically and thus has become a fundamental pillar for decentralized optimization over directed networks. We next analyze the influence of $W$ on the convergence rate of push-sum.

\begin{figure} 
     \centering
     \includegraphics[width=2.2in]{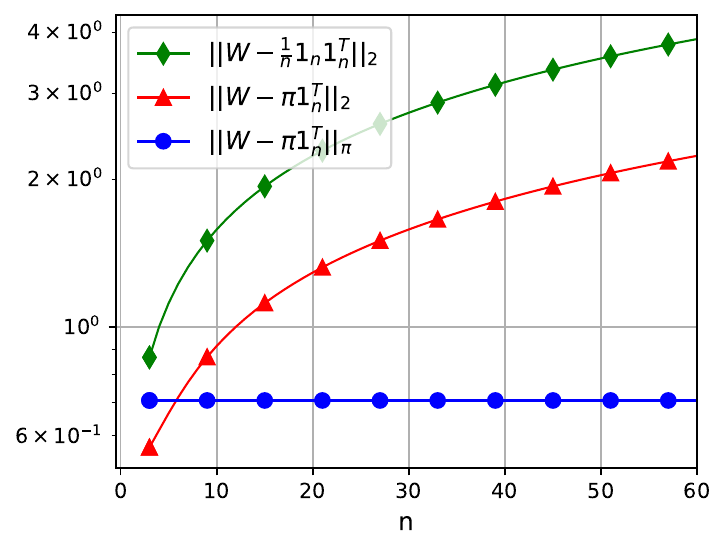}
     \vspace{-5mm}
     \caption{Comparison of the canonical spectral gap and the generalized spectral gap of the column-stochastic instance in Sec.~\ref{appendix:examples}.}
     \label{fig:2-norm v.s. pi-norm}
     \vspace{-0.3cm}
 \end{figure}

\subsection{Generalized spectral gap} 
For undirected graphs, the spectral gap $1-\beta$ with $\beta:=\|W-\mathds{1}_n \mathds{1}_n^\top/n\|_2$ measures the performance of partial averaging. However, $\beta$ can exceed $1$ for ill-conditioned digraphs, making it ineffective for measuring the performance of push-sum as it no longer admits contraction; see the illustration in Figure~\ref{fig:2-norm v.s. pi-norm}.
Fortunately, it is proved in \cite{xin2019SAB} that $\|W - \pi\mathds{1}_n^\top\|_\pi$ is always below $1$ for  matrix $W$ satisfying Assumption~\ref{ass-weight-matrix}, where  $\|A\|_\pi:=\|\mathrm{diag}(\sqrt{\pi})^{-1}A \ \mathrm{diag}(\sqrt{\pi})\|_2$ for $A\in\RR^{n\times n}$. Moreover, when $\pi=\one_n/n$, \ie, $W$ is doubly-stochastic, the matrix $\pi$-norm $\|\cdot\|_{\pi}$ reduces to the regular $2$-norm. Therefore, we can naturally view $1-\|W - \pi\mathds{1}_n^\top\|_\pi$ as a generalization of the canonical spectral gap $1-\beta$ defined for doubly-stochastic matrices with undirected networks. Formally, we have the following definition.
\begin{definition}[\sc Generalized Spectral gap~\cite{xin2019SAB}]\label{def:beta}
Consider a matrix $W$ satisfying Assumption~\ref{ass-weight-matrix} and let $\pi$ be its equilibrium vector. We define its generalized spectral gap as $1-\beta_\pi$ (abbreviated as spectral gap hereafter) where
\begin{align}\label{beta definition}
\beta_\pi:=\|W - \pi\mathds{1}_n^\top\|_\pi\in[0,1).
\end{align}
\end{definition}

\subsection{Equilibrium skewness} { The generalized spectral gap can reflect the convergence of \pushsum in certain scenarios as illustrated in the left plot in Figure~\ref{fig:push-sum1,2}. However, the right plot of Figure~\ref{fig:push-sum1,2} shows that the identical spectral gap $1-\beta_\pi$ may also lead to different convergence rates. This simulation result conveys a clear message:}
\emph{the spectral gap alone is insufficient to characterize the convergence rate of decentralized algorithms over general directed networks.} 
To identify additional relevant metrics, we consider a specific version of \pushsum where $v^{(0)}=n\pi$. In this special case, we have $v^{(k)}=n\pi$, $\forall k\geq0$, and thus \pushsum iterations follow
\begin{equation*}
    \vz\ \xmapsto{\text{power iterations}}\ W^k\vz \ \xmapsto{\text{one-shot reweighting}} \ \mathrm{diag}(n\pi)^{-1}W^k\vz=:\vw^{(k)}.
\end{equation*}
Since $\vw^{(k)}-\bar \vz= \mathrm{diag}(n\pi)^{-1}(W^k\vz-\pi\one_n^\top \vz)=\mathrm{diag}(n\pi)^{-1}(W-\pi\one_n^\top)^k \vz$, we have
\vspace{-2pt}
\begin{align}
    \|\vw^{(k)}-\bar \vz\|_F^2= &{\sum_{j=1}^d}\Big\|\mathrm{diag}(n\pi)^{-1}(W-\pi\one_n^\top)^k \vz_{\cdot,j}\Big\|_2^2\nonumber\\
    \leq &\Big\|\mathrm{diag}(n\pi)^{-1}\mathrm{diag}(\sqrt{\pi})\Big\|_2^2 {\sum_{j=1}^d}\|(W-\pi\one_n^\top)^k \vz_{\cdot,j}\|_{\pi}^2\nonumber\\
    \leq &\frac{\beta_\pi^{2k}{\sum_{j=1}^d}\| \vz_{\cdot,j}\|_{\pi}^2}{n^2 \min_i \pi_i}
    \leq  \frac{\max_i \pi_i^2}{ \min_i \pi_i^2}\beta_\pi^{2k} \|\vz\|_{F}^2.\label{eqn:vhdifsnvxz}
\end{align}
From \eqref{eqn:vhdifsnvxz}, we observe that the squared error of \pushsum iterations stems from both the power iterations, governed by $\beta_\pi^k$, and the reweighting step that corrects the disparity between the $\pi$-norm and $l_2$ norm, governed by $\max_i\pi_i/\min_i \pi_i$. The constant $\max_i\pi_i/\min_i \pi_i$ captures the disagreement between $\pi$ and $\mathds{1}_n/n$, and represents the skewness of the equilibrium vector $\pi$. We formally have the following definition.

\begin{figure}[t]
    \centering
\subfigure{
\begin{minipage}[t]{0.45\textwidth}
\centering

\makebox[\textwidth][c]{\includegraphics[width=2.2in]{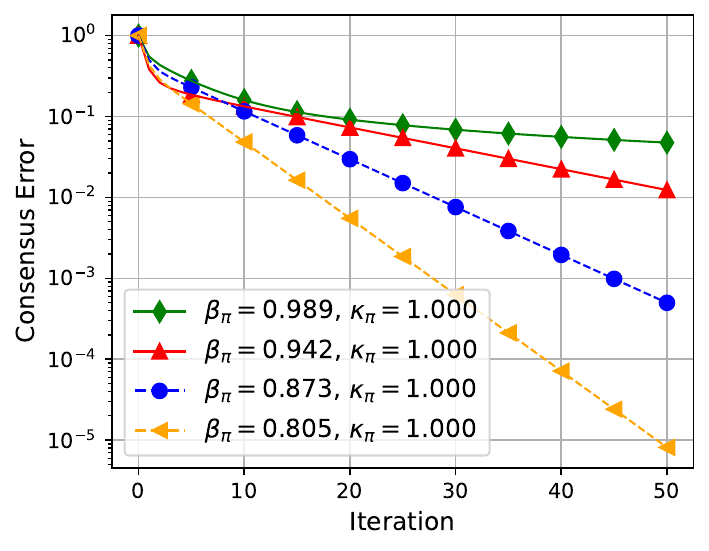}}
\end{minipage}
}
    \subfigure{
\begin{minipage}[t]{0.45\textwidth}
\centering

\makebox[\textwidth][c]{\includegraphics[width=2.2in]{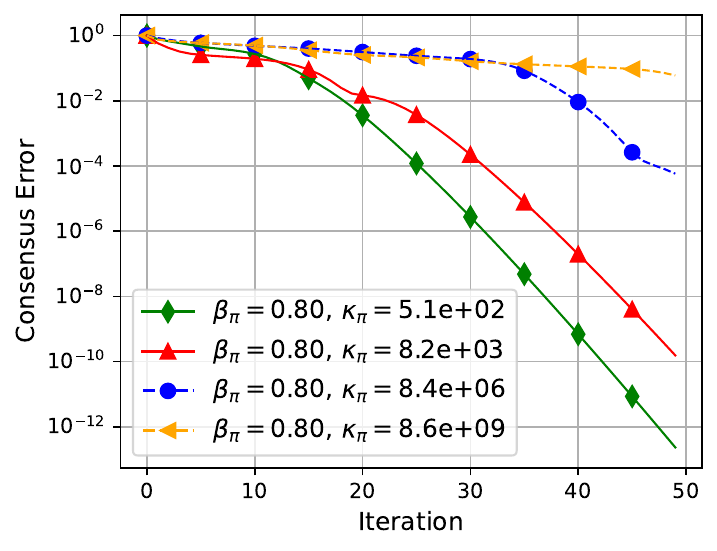}}

\end{minipage}
} \vspace{-5mm}
    \caption{Push-sum iterations with matrices of different spectral gaps and skewness. The y-axis representing the relative consensus error, $\|\vz^{(k)}-\bar \vz\|_F/\|\vz^{(0)}-\bar \vz\|_F$. The {left} figure illustrates the performance of \pushsum with identical $\kappa_\pi$ but varying $\beta_\pi$.  The {right} figure illustrates performance of \pushsum with identical $\beta_\pi$ but varying $\kappa_\pi$.}
    \label{fig:push-sum1,2}
\vspace{-0.4cm}
\end{figure}

\begin{definition}[\sc Equilibrium skewness]
   Let $W$ satisfy Assumption~\ref{ass-weight-matrix} and $\pi$ be its equilibrium vector. The equilibrium skewness of $W$ is then defined as
   \vspace{-2pt}
    $$\kappa_\pi:={\mathop{\max}_{i}\pi_i}/{\mathop{\min}_{i}\pi_i}\in[1,+\infty).$$ 
\end{definition}
It is worth noting that the equilibrium skewness does not manifest in decentralized algorithms over undirected networks because $\pi = n^{-1}\mathds{1}_n$ and $\kappa_\pi = 1$ in this case.

\subsection{Network influence on push-sum}
By jointly considering both the spectral gap and equilibrium skewness, we are able to provide an analysis for \pushsum which fully reveals the influence of directed networks in  Lemma~\ref{push-sum lemma}. This result also facilitates our subsequent analysis of decentralized optimization algorithms.

\begin{lemma}\label{push-sum lemma}
Under Assumption~\ref{ass-weight-matrix},
the following statements hold for \pushsum iterations \eqref{eqn:push-sum}  with initialization $\vz^{(0)}$ and $v^{(0)}$ such that $\one_n^\top v^{(0)}=n$.
 \begin{enumerate}
     \item $\min_{i}  \{{v_i^{(k)}}/{\pi_i}\}$ is a non-decreasing sequence with respect to $k$.
     \item $\max_{i}  \{{v_i^{(k)}}/{\pi_i}\}$ is a non-increasing sequence with respect to $k$.
     \item $\big\|{V^{(k)}}^{-1}\big\|_2 \le \kappa_\pi $.   
     \item $\|\vw^{(k)}-\bar \vz^{(0)}\|_F\le \kappa_\pi^{1.5}\beta_\pi^k\|\vz^{(0)}\|_F$. 
 \end{enumerate}
 \label{small-lemma}
\end{lemma}

Recall that \pushsum estimates the desired average $\bar\vz^{(0)}$ by ${V^{(k)}}^{-1}W^k \vz^{(0)}$. The third statement in Lemma~\ref{push-sum lemma} bounds the norm of the matrix inverse ${V^{(k)}}^{-1}$ by $\kappa_\pi$. On the other hand, prior work \cite{Olshevsky2015Dist,Nedic2017pushdiging,9053238} typically bounds $\big\|{V^{(k)}}^{-1}\big\|_2$ by an extremely conservative value, \eg, $\big\|{V^{(k)}}^{-1}\big\|_2\le n^n$, which provides limited insights on the influence of directed topology on the convergence of decentralized algorithms.


\subsection{Influence of \texorpdfstring{$\kappa_\pi$}{} can be substantial}
As illustrated in Lemma~\ref{push-sum lemma}, achieving $\|\vw^{(k)}-\bar \vz^{(0)}\|_F=O(\epsilon)$ with \pushsum requires $O(\ln(\kappa_\pi/\epsilon)/(1-\beta_\pi))$ communication rounds. As the influence of the equilibrium skewness $\kappa_\pi$ is only logarithmic in the convergence rate, one might naturally speculate whether its effect is negligible in comparison to the spectral gap $1-\beta_\pi$. We negate this speculation by showing that $\kappa_\pi$ can be exponentially large with respect to the network size $n$.
\begin{proposition}\label{prop:special-mat}
    For any $n\geq 1$, there exists a  matrix $W\in\RR^{n\times n}$ satisfying Assumption~\ref{ass-weight-matrix} such that $\beta_\pi=\frac{\sqrt{2}}{2}$ but $\kappa_\pi=2^{n-1}$.
\end{proposition}

Proposition~\ref{prop:special-mat} shows that the impact of the equilibrium skewness can dominate the spectral gap in convergence rates, even after taking the logarithm, since $\ln(\kappa_\pi)=\Omega(n)$ while $(1-\beta_\pi)^{-1}$ is $\mathcal{O}(1)$. This observation highlights the central role of the equilibrium skewness in decentralized optimization over general directed networks. We verify these theoretical insights by numerical results in the right plot of Figure~\ref{fig:push-sum1,2}, where the convergence of \pushsum varies significantly for matrices with (nearly) identical spectral gaps but different equilibrium skewness.

Another important implication in Proposition~\ref{prop:special-mat} is that the metrics $\kappa_\pi$ and $1-\beta_\pi$ can be orthogonal over certain directed networks, with $1-\beta_\pi$ being a constant while $\kappa_\pi$ grows exponentially; {see the illustration in the right plot of Figure~\ref{fig:intro-2}}. This again necessitates the need to consider both metrics when evaluating the influence of directed networks on decentralized algorithms.

\section{Lower bound for decentralized optimization}\label{sec:lower-bound}
In Sec.~\ref{sec:prelim}, we show that the spectral gap~$1-\beta_\pi$ and the equilibrium skewness $\kappa_\pi$ are sufficient to characterize the convergence rate of \pushsum and optimization algorithms built upon it. However, it remains unclear if they are necessary for general decentralized optimization. We address this issue by establishing a lower bound for smooth and non-convex decentralized stochastic
optimization using column-stochastic mixing matrices. 

\subsection{Assumptions}\label{sec:assumptions}
In this subsection, we describe the class of decentralized optimization problems to which our lower bound applies.

\vspace{1mm}
\noindent \textbf{Function class.} We let the {function class $\cF_{\Delta, L}$} denote the set of functions satisfying Assumption~\ref{ass:smooth} for any underlying dimension $d\in \mathbb{N}_+$ and initialization $x^{(0)}\in\RR^d$.
\begin{assumption}[\sc Smoothness] \label{ass:smooth}
	There exists a constant $L,\Delta\geq 0$ such that
	$$
	\left\|\nabla f_{i}(x)-\nabla f_{i}(y)\right\| \leq L\|x-y\|
	$$
	for all $1\leq i\leq n, x, y \in \mathbb{R}^{d}$, and $f(x^{(0)})-\inf _{x \in \mathbb{R}^{d}} f(x)\leq \Delta $.
\end{assumption}

\vspace{1mm}
\noindent \textbf{Gradient oracle class.} 
We assume each node $i$ processes its local cost~$f_i$ by means of a stochastic gradient oracle $\nabla F(x;\xi_i)$ that admits unbiased estimates of the exact gradient $\nabla f_i$ with bounded variance. Formally, we let the {stochastic gradient oracle class $O_{\sigma^2}$} denote the set of all oracles $\nabla F(\cdot;\xi_i)$ that satisfy Assumption~\ref{ass:sto}.
\begin{assumption}[\sc Gradient oracles]\label{ass:sto}
There exists a constant $\sigma\geq 0$ such that  
	\begin{align*}
		\mathbb{E}[\nabla F(x;\xi_i)] = \nabla f_i(x), \; \mathbb{E}[\|\nabla F(x;\xi_i) - \nabla f_i(x)\|^2] \le \sigma^2, \;\forall\,x\in \mathbb{R}^d,~\forall \,1\leq i \leq n. 
	\end{align*} 
\end{assumption}

\vspace{1mm}
\noindent \textbf{Algorithm class.} The decentralized network of interest is a strongly-connected, directed graph of $n$ nodes with a column-stochastic mixing matrix $W=[w_{ij}]_{i,j=1}^n\in\RR^{n\times n}$. We focus on decentralized algorithms that satisfy Assumption~\ref{ass:sto}. In particular, each node $i$ maintains a local solution $x_i^{(k)}$ at iteration $k$ and follows the partial averaging policy, \ie, nodes communicate simultaneously via the $W$-protocol as \eqref{eqn:w-protocol}. In addition, we assume that algorithms in discussion follow the linear-spanning property~\cite{carmon2020lower,carmon2021lower,yuan2022revisiting,lu2021optimal}. Informally speaking, the linear-spanning property requires that the local solution $x_i^{(k)}$ lies in the linear space spanned by $x_i^{(0)}$, its local stochastic
gradients, and interactions with the neighboring nodes. Furthermore, upon the end of $K$ algorithmic iterations, we allow the ultimate output $\hat x^{(K)}$ to be any variable in $\mathrm{span}(\{\{x_i^{(k)}\}_{i=1}^n\}_{k=0}^K)$.
The linear-spanning condition is
met by all methods in Table~\ref{tab:compare} and most first-order optimization methods.
We let $\cA_{W}$
be the set of all algorithms following the partial averaging (through matrix $W$)  and the linear-spanning property.

\subsection{Lower Bound}
We are now ready to state our lower bound.
\begin{theorem}\label{thm:lower-bound}
    For any given $L\geq 0$, $n\geq 2$, $\sigma\geq 0$, and  $\tilde \beta\in[\Omega(1), 1-1/n]$, there exists a set of loss functions $\{f_i\}_{i=1}^n\in \cF_{\Delta,L}$, a set of stochastic gradient oracles in $\cO_{\sigma^2}$, and a column-stochastic  matrix $W\in\RR^{n\times n}$ with 
$\beta_\pi=\tilde \beta$  and $\ln(\kappa_\pi)=\Omega(n(1-\beta_\pi))$,
such that the convergence of any $A\in\cA_W$ starting form $x_i^{(0)}=x^{(0)}$, $\forall\, 1\leq i\leq n$ with $K$ iterations is lower bounded by 
\vspace{-2pt}
    \begin{align}\label{eqn:lower-bound}
        \EE[\|\nabla f(\hat x^{(K)})\|_2^2]=\Omega\left(\frac{ \sigma \sqrt{L\Delta}}{\sqrt{nK}
        } + \frac{(1+\ln(\kappa_\pi))L\Delta}{(1-\beta_\pi)K}\right).
    \end{align}
\end{theorem}

\vspace{-2pt}
In view of Theorem~\ref{thm:lower-bound}, a few remarks are in place.


\vspace{1mm}
\noindent \textbf{Necessity of $\kappa_\pi$ for directed networks.}
The lower bound shows that {\em the equilibrium skewness $\kappa_\pi$ is a necessary notion beyond the spectral gap $1-\beta_\pi$ in measuring the performance of decentralized algorithms over directed networks.} Moreover, while $\kappa_\pi$ itself can be exponentially large with respect to the network size $n$ as presented in Proposition~\ref{prop:special-mat}, the lower bound implies its impact can be possibly mitigated logarithmically to $\ln(\kappa_\pi)$, perhaps with the help of proper algorithmic designs.

\begin{table}[!t]
    \caption{Comparison of lower bounds in decentralized stochastic smooth non-convex optimization with a certain type of mixing matrix $W$.}
    \centering
    \begin{tabular}{l c c}
        Literature & Lower Bound  & Mixing Matrix  \\
        \toprule
        \cite{yuan2022revisiting,lu2021optimal}&  $\frac{ \sigma \sqrt{L\Delta}}{\sqrt{nK}
        } + \frac{L\Delta}{\sqrt{1-\beta}K}$ & \makecell[c]{Doubly-stochastic \& static}\\
        \cite{huang2022optimal}&  $\frac{ \sigma \sqrt{L\Delta}}{\sqrt{nK}
        } + \frac{L\Delta}{(1-\beta)K}$ &  \makecell[c]{Doubly-stochastic \& time-varying}\\
        \midrule
        Theorem~\ref{thm:lower-bound}&  $\frac{ \sigma \sqrt{L\Delta}}{\sqrt{nK}
        } + \frac{(1+\ln(\kappa_\pi))L\Delta}{(1-\beta_\pi)K}$ & \makecell[c]{Column-stochastic \& static}\\
        \bottomrule
    \end{tabular}
    \vspace{-0.5cm}
    \label{tab:lower-bounds}
\end{table}

\vspace{1mm}
\noindent \textbf{Comparison with other decentralized lower bounds.}
The lower bound in Theorem~\ref{thm:lower-bound} is novel compared to the existing lower bounds in decentralized optimization with doubly-stochastic mixing weights as it reflects the impacts of the spectral gap $1-\beta_\pi$ and the equilibrium skewness  $\kappa_\pi$ simultaneously. See the comparison in Table~\ref{tab:lower-bounds}. {It is observed that the lower bound on decentralized optimization increases as the condition on network topology gets relaxed.}

\vspace{1mm}
\noindent \textbf{Comparison with the centralized lower bound and transient time.} 
We note that the first term in \eqref{eqn:lower-bound} matches the oracle lower bound for centralized smooth stochastic non-convex optimization~\cite{carmon2020lower}, which is attained by centralized parallel SGD, while the second term therein reflects the impact of directed network. Furthermore, the first term dominates the second term when $K$ is sufficiently large. In other words, our lower bound implies that an optimal algorithm, if exists, for this class of decentralized problems matches the performance of centralized parallel SGD after a finite number of transient iterations. Specifically, this requires $(\frac{L\Delta \sigma^2}{nK})^{{1}/{2}} \gtrsim \frac{(1+\ln(\kappa_\pi))L\Delta}{(1-\beta_\pi)K}$, \ie, $K\gtrsim n(1+\ln(\kappa_\pi))^2/(1-\beta_\pi)^2$ to bypass {the influence of} decentralized topology.

\section{Push-DIGing algorithms}\label{sec-push-sum based algorithms}
In this section, we revisit the well-known \pushdg algorithm~\cite{Nedic2017pushdiging} and establish a refined convergence rate with the help of the generalized spectral gap and the equilibrium skewness. We next integrate Push-DIGing with multi-round \pushsum
to attain optimal complexity.

\subsection{Vanilla Push-DIGing algorithm}
\pushdg~\cite{Nedic2017pushdiging} combines \pushsum with the standard gradient tracking update~\cite{qu2017harnessing} to obtain the following updates:
\vspace{-2pt}
\allowdisplaybreaks
\begin{align}
\vx^{(k+1)}&=W(\vx^{(k)}-\gamma \vy^{(k)}) \nonumber\\
        v^{(k+1)} &=W v^{(k)}  \nonumber\\
    V^{(k+1)}&=\mathrm{diag}(v^{(k+1)}) \label{Push-DIGing iteration}\\
    \vw^{(k+1)}&={V^{(k+1)}}^{-1}\vx^{(k+1)} \nonumber\\
    \vy^{(k+1)}&=W(\vy^{(k)}+\nabla F(\vw^{(k+1)};\bxi^{(k+1)})-\nabla F(\vw^{(k)};\bxi^{(k)})) \nonumber
\end{align}
where $\vw^{(0)}=\vx^{(0)}$, $\vy^{(0)}=\nabla F(\vx^{(0)}; \bxi^{(0)})$,  and $v^{(0)}=\mathds{1}_{n}$.
The next theorem presents the convergence rate of \pushdg with explicit dependence on $1-\beta_\pi$ and $\kappa_\pi$.

\begin{theorem}[\sc Convergence of Push-DIGing] \label{thm:push diging convergence}
Under Assumptions \ref{ass-weight-matrix}, \ref{ass:smooth}, and \ref{ass:sto}, by setting step-size $\gamma$ as in \eqref{main theorem eq-1}, \pushdg converges as
\vspace{-2pt}
\begin{align}\label{eqn:jisdnfzxv}
\frac{1}{K+1}\sum_{k=0}^K\EE\left[\left\|\nabla f(\bar{x}^{(k)})\right\|_2^2\right] \lesssim \frac{\sigma}{\sqrt{nK}}+\frac{\beta_\pi^{{4}/{3}}\kappa_\pi^{2} \sigma ^{{2}/{3}}}{(1-\beta_\pi)K^{{2}/{3}}}+\frac{\kappa_\pi^7\beta_\pi}{(1-\beta_\pi)^2K}+\frac{1}{K}.
\end{align}

where $L$, $\Delta$, and $\fnorm{\vy^{(0)}}^2$ are absorbed to highlight the network influence. 
\end{theorem}

\vspace{1mm}
\noindent \textbf{Comparison with existing analyses.}
Theorem~\ref{thm:push diging convergence} presents the first {finite-time} non-convex convergence result for \pushdg, while previous analyses either focus on the strongly convex case~\cite{Nedic2017pushdiging,liu2024distributed} or the asymptotic rate~\cite{kungurtsev2023decentralized}. 
{Furthermore, Theorem~\ref{thm:push diging convergence} firstly clarifies the impact of directed network on \pushdg with $1-\beta_\pi$ and $\kappa_\pi$.} When the mixing matrix $W$ is doubly-stochastic, \ie, $\kappa_\pi=1$, our rate recovers that of stochastic gradient tracking for smooth non-convex problems~\cite{GT-Xinran}.

\vspace{1mm}
\noindent \textbf{Transient time.} To reach the convergence rate $\sigma/\sqrt{nK}$ of centralized parallel SGD, \pushdg requires the number of iterations $K$ to satisfy 
\vspace{-2pt}
\begin{equation*}
    \max\left\{\frac{\beta_\pi^{{4}/{3}}\kappa_\pi^2}{(1-\beta_\pi)K^{{2}/{3}}},  \frac{\kappa_\pi^7\beta_\pi}{(1-\beta_\pi)^2K},\frac{1}{K}\right\}\lesssim \frac{1}{\sqrt{nK}},
\end{equation*} \ie, $K\gtrsim{n^3\kappa_\pi^{14}}/{(1-\beta_\pi)^6}$. We note that in comparison to the optimal transient time $n(1+\ln(\kappa_\pi))^2/(1-\beta_\pi)^2$ suggested by Theorem~\ref{thm:lower-bound}, there is a significant gap.

\subsection{MG-Push-DIGing algorithm}\label{sec-MG push-sum type algorithm}
In this subsection, we develop an optimal algorithm to attain the lower bound~\eqref{eqn:lower-bound}. Inspired by the algorithm development in \cite{lu2021optimal,huang2022optimal}, two additional components are added to \pushdg: gradient accumulation and multiple-gossip (MG) communication, thus giving the name \mgpushdg. 
In particular, we replace every $W$ with $W^R$ in the \pushdg updates. In other words, each node takes $2R$ rounds of communication at $k$-th iteration instead of two rounds of communication in vanilla \pushdg. In addition, we further reduce the variances of stochastic gradients by making the sampling batch $R$ times larger. {The \mgpushdg method is listed in Algorithm~\ref{algorithm:MG-Push-DIGing}, and the following theorem establishes its convergence rate.}

\begin{algorithm}[t]
\caption{\mgpushdg: \pushdg with multi-round gossip} \label{algorithm:MG-Push-DIGing}
\begin{algorithmic}
\REQUIRE Initialize $v^{(0,0)}=\mathds{1}_n$, $\vw^{(0)}=\vx^{(0)}$, $g_i^{(0)}=y_i^{(0)}=\frac{1}{R}\sum_{r=1}^{R}\nabla F(x^{(0)};\xi_i^{(0,r)})$, the mixing matrix $W$, the multi-round number $R$.
\FOR{$k=0,1,\dots,K-1$, each node $i$ in parallel}
\STATE Let $\phi^{(k+1,0)}=x_i^{(k)}-\gamma y_i^{(k)}$;
\FOR{$r=0,1,\dots,R-1$, each node $i$ in parallel}
\STATE Update $\phi^{(k+1,r+1)}_i=\sum_{j\in \mathcal{N}_i^{\mathrm{in}} } w_{ij}\phi^{(k+1,r)}_j $ { and $v^{(k,r+1)}_i=\sum_{j\in \mathcal{N}_i^{\mathrm{in}} } w_{ij} v^{(k,r)}_j$};
\ENDFOR
\STATE Update $x_i^{(k+1)}=\phi^{(k+1,R)}_i$, {$v^{(k+1,0)}_i=v^{(k,R)}_i$ and $w^{(k+1)}_i=\phi^{(k+1,0)}_i/ v^{(k+1,0)}_i$};
\STATE Compute $g_i^{(k+1)}=\frac{1}{R}\sum_{r=1}^{R}\nabla F(w_i^{(k+1)};\xi_i^{(k+1,r)})$;
\STATE Let $\psi^{(k+1,0)}_i=y^{(s)}_i+g^{(k+1)}_i-g^{(s)}_i$;
\FOR{$r=0,1,\dots,R-1$, each node $i$ in parallel} 
\STATE Update $\psi^{(k+1,r+1)}_i=\sum_{j\in \mathcal{N}_i^{\mathrm{in}} } w_{ij}\psi^{(k+1,r)}_j $;
\ENDFOR
\STATE Update $y^{(k+1)}_i=\psi_i^{(k+1,R)}$; 
\ENDFOR
\end{algorithmic}
\end{algorithm}

\begin{theorem}[\sc Convergence of MG-Push-DIGing]
\label{thm:MG-Push-DIGing convergence}
Under Assumptions \ref{ass-weight-matrix}, \ref{ass:smooth}, and \ref{ass:sto}, by setting $R=\lceil\frac{(1+\sqrt{7\ln(\kappa_\pi)})^2+(1+\sqrt{2\ln(n)})^2}{1-\beta_\pi}\rceil$, ${T=(K+1)R}$, and $\gamma$ adapted from \eqref{main theorem eq-1}, \mgpushdg converges as
    \begin{align}\label{eqn:mg-upper}
        \frac{1}{K+1}\sum_{k=0}^{K}\EE\left[\left\|\nabla f(\bar{x}^{(k)})\right\|_2^2\right]=\widetilde{\mathcal{O}}\left(\frac{\sigma\sqrt{L\Delta}}{\sqrt{nT}}+\frac{(1+\ln(\kappa_{\pi}))L\Delta}{(1-\beta_\pi)T}\right),
    \end{align}
    where $\widetilde{\mathcal{O}}(\cdot)$ absorbs logarithmic factors independent of $\kappa_\pi$, $\beta_\pi$, and $T$.
\end{theorem}

Remarkably, the rate \eqref{eqn:mg-upper} matches with the lower bound \eqref{eqn:lower-bound} up to logarithmic factors that are independent of $\kappa_\pi$ and $\beta_\pi$. This shows the near optimality of \mgpushdg and the tightness of the lower bound \eqref{eqn:lower-bound}. A comparison of \mgpushdg with other state-of-the-art algorithms is presented in Table~\ref{tab:compare}. 

\section{Experiments}
\label{sec:experiments}
We present  simulations to justify our theoretical findings.

\subsection{A family of highly skewed networks}\label{appendix:examples} In this subsection, we present a column-stochastic matrix whose skewness $\kappa_\pi$ is large while the spectral gap $1-\beta_\pi$ is small.
We consider the following column-stochastic matrix
\begin{equation}\label{eqn:bad-W}
    W=\begin{bmatrix}
1/2 & 1/2 &\cdots &1/2 & 1\\
1/2 & 0 & & & \\
 &\ddots  & \ddots &  & \\
 &  & 1/2 &0 &\\
 &  & & 1/2&0\\
\end{bmatrix}\in \RR^{n\times n}.
\end{equation}
The equilibrium vector of $W$ is $\pi\propto \left(2^{n-1},2^{n-2},\dots,1\right)^\top$, resulting in $\kappa_\pi=2^{n-1}$. Moreover, it can be shown that $\beta_\pi\equiv1/\sqrt{2}$ for $W$ {of an arbitrary size}, see Appendix~\ref{app:special-mat} for the proof. The underlying network associated with $W$ requires $n-1$ steps to transmit information from node $1$ to node $n$, see the left plot of Figure~\ref{fig:intro-2} for an illustration when $n=7$.
For such a matrix, its equilibrium skewness  $\kappa_\pi$ (even taking the logarithm $\ln(\kappa_\pi)$) is significantly larger than $(1-\beta_\pi)^{-1}$, see the right plot of Figure~\ref{fig:intro-2}.

\subsection{Influence of digraph on Push-DIGing}  

In this subsection, we examine the performance of \pushdg.
Our numerical experiments are based on a decentralized logistic
regression problem with non-convex regularization~\cite{antoniadis2011penalized,GT-Xinran,alghunaim2022unified} over a directed graph. In particular, we consider $ \min_{x\in \mathbb{R}^d} n^{-1}\sum_{i=1}^n (f_i(x)+\rho r(x))$, where
\vspace{-2pt}
\begin{equation*}
f_i(x)=\frac{1}{M}\sum_{l=1}^M \ln(1+\exp(-y_{i,l}h_{i,l}^\top x))\quad \text{and} \quad r(x)=\sum_{j=1}^d\frac{[x]_j^2}{1+[x]_j^2}.
\vspace{-0.1cm}
\end{equation*}
Here, $[x]_j$ denotes the $j$-th entry of $x$, $\{h_{i,l},y_{i,l}\}_{l=1}^M$ is the training dataset held by node $i$ in which $h_{i,l} \in \mathbb{R}^d$ is a feature vector while $y_{i,l} \in \{+1,-1\}$ is the corresponding label. The regularization $r(x)$ is non-convex and the parameter $\rho>0$ controls its strength.

\subsubsection{Experimental Settings} 
We set $d = 10$, $M = 2000$, and $\rho = 0.001$. To account for data heterogeneity across nodes, each node $i$ is associated with a local solution $x_i^{\star}$. Such $x_i^{\star}$ is generated by $x_i^{\star} = x^{\star}+v_i$, where the shared $x^{\star}\sim \mathcal{N}(0,I_d)$ is randomly generated, while $v_i\sim \mathcal{N}(0,\sigma_h^2I_d)$ controls the similarity between local solutions.
With $x_i^{\star}$ at hand, we can generate local data that follows distinct distributions. At node $i$, we generate each feature vector $h_{i,l} \sim \mathcal{N} (0, I_d)$. To produce the corresponding label $y_{i,l}$, we generate a random variable $z_{i,l}\sim \mathcal{U}(0, 1)$. If $z_{i,l} <1/(1+\mathrm{exp}(-y_{i,l}h_{i,l}^{\top} x_i^{\star}))$, we set $y_{i,l}=1$; otherwise $y_{i,l}=-1$. Clearly, solution $x_i^{\star}$ controls the distribution of the labels. To control gradient noise, we introduce Gaussian noise to the actual gradient to achieve stochastic gradients. Specifically, we have $\nabla f_i(x) = \nabla f(x) + \varepsilon_i$, where $\varepsilon_i\sim \mathcal{N}(0,\sigma_n^2 I_d)$. We can control the magnitude of the gradient noise by adjusting $\sigma_n^2$. Throughout the experiments, we set $\sigma_h=1$, $\sigma_n=0.001$. The performance metric of interest across all experiments is $\|\nabla f(\bar{x}^{(k)})\|$.

\subsubsection{Convergence with varying \texorpdfstring{$\beta_\pi$}{}} 
To explore how $1-\beta_\pi$ affects convergence, we create multiple column-stochastic matrices with the same topology as in Figure~\ref{fig:intro-2}. We carefully choose these matrices to have the equilibrium skewness $\kappa_\pi \in (162, 164)$, while  $1/(1-\beta_\pi)$ varies from $5.0$ to $695.5$. Notably, the equilibrium vector $\pi$ can be different across these matrices.
By running \pushdg with these mixing matrices, we gain direct insight into how the spectral gap $1-\beta_\pi$ influences convergence rates. The results are depicted in the left plot of Figure~\ref{fig:diging}, which shows that the convergence slows down as $\beta_\pi$ gets closer to $1$, \ie, the spectral gap gets larger. {This observation is consistent with our results in Theorem~\ref{thm:push diging convergence}.}

\subsubsection{Convergence with varying \texorpdfstring{$\kappa_\pi$}{}} 
Likewise, we generate multiple column-stochastic matrices with the same topology shown in Figure~\ref{fig:intro-2} with nearly the same inverse spectral gap $(1-\beta_\pi)^{-1} \in (10,10.1)$. In contrast,  their skewness $\kappa_\pi$ varies from $4.1$ to $3250.9$. By running \pushdg with these mixing matrices, we observe that $\kappa_\pi$ affects convergence rates significantly. The results are depicted in the right plot of Figure~\ref{fig:diging}, which demonstrates that the convergence slows down as $\kappa_\pi$ increases. {This observation is consistent with Theorem~\ref{thm:push diging convergence}.}

\begin{figure}[htbp]
    \centering
\subfigure{
\begin{minipage}[t]{0.45\textwidth}
\centering

\makebox[\textwidth][c]{\includegraphics[width=2.2 in]{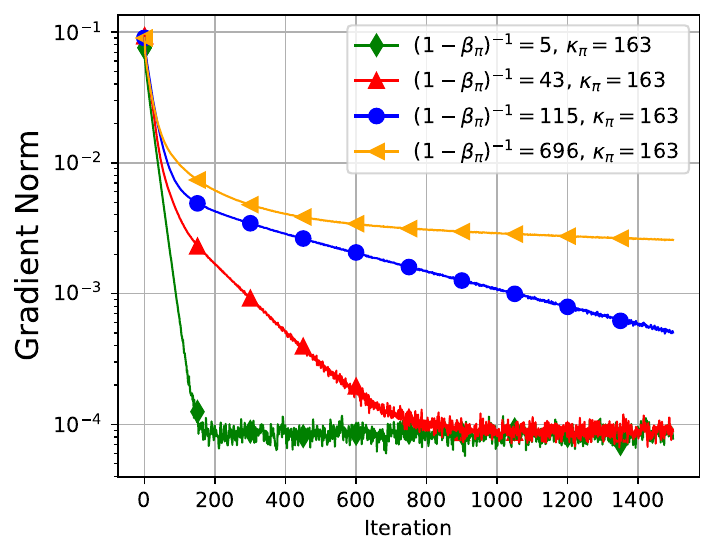}}

\end{minipage}
}
    \subfigure{
\begin{minipage}[t]{0.45\textwidth}
\centering

\makebox[\textwidth][c]{\includegraphics[width=2.2 in]{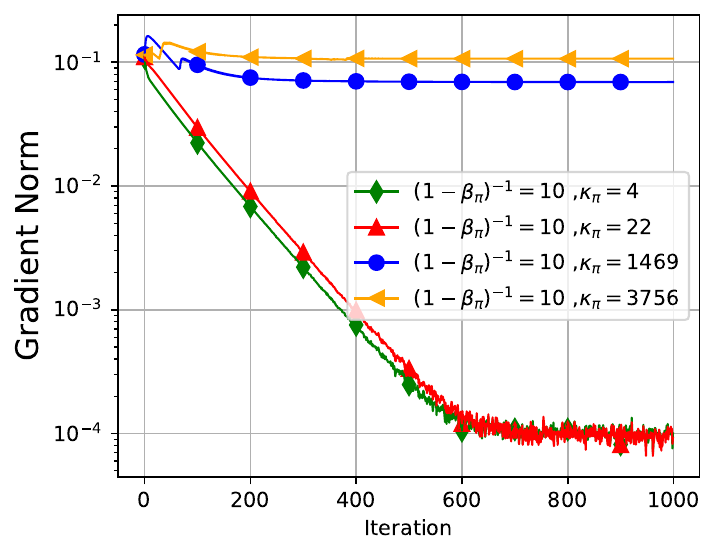}}

\end{minipage}
} 
\vspace{-5mm}
    \caption{The performance of Push-DIGing under different setups. (Left.) $1-\beta_\pi$ varies while $\kappa_\pi\approx 163$ is fixed. (Right.)  $\kappa_\pi$ varies while $(1-\beta_\pi)^{-1}\approx 10$ is fixed. }
    \label{fig:diging}
\vspace{-0.6cm}
\end{figure}
\begin{figure}[htbp]
    \centering
\subfigure{
\begin{minipage}[t]{0.45\textwidth}
\centering

\makebox[\textwidth][c]{\includegraphics[width=2.2 in]{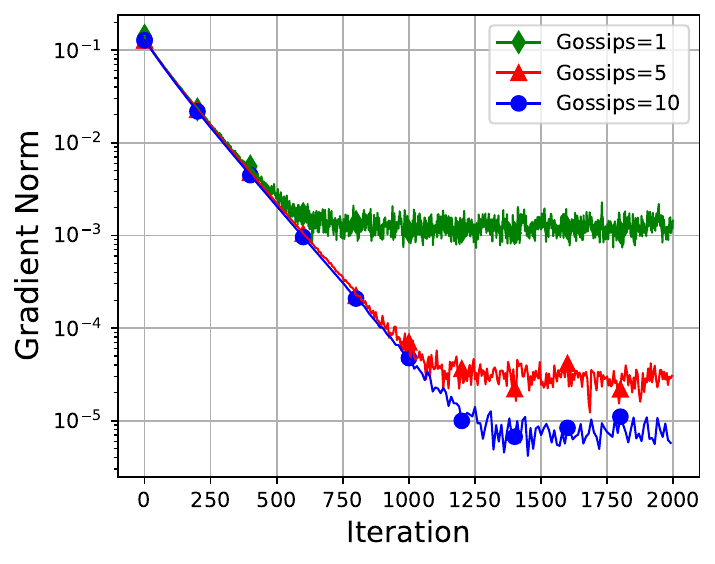}}

\end{minipage}
}
    \subfigure{
\begin{minipage}[t]{0.45\textwidth}
\centering

\makebox[\textwidth][c]{\includegraphics[width=2.2in]{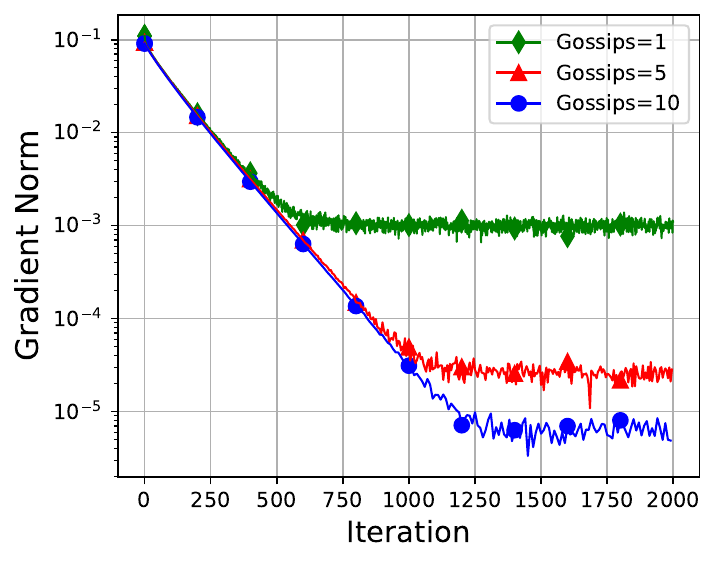}}

\end{minipage}
} 
\vspace{-5mm}
    \caption{MG-Push-DIGing. For $W_1$ used in the left plot, $\kappa_\pi=4804.49$, $(1-\beta_\pi)^{-1}=2.34$. For $W_2$ used in the right plot, $\kappa_\pi=6.33$, $(1-\beta_\pi)^{-1}=51.24$.}
    \label{fig:MG diging}
\vspace{-0.5cm}
\end{figure}

\subsection{Comparison of MG-Push-DIGing and Push-DIGing}\label{appendix:MG DIGing experiement}
In this subsection, we numerically justify the benefit of multi-round gossip by 
comparing \mgpushdg and the vanilla \pushdg. According to Theorem~\ref{thm:MG-Push-DIGing convergence}, the MG technique effectively mitigates the impact of both $1-\beta_\pi$ and $\kappa_\pi$. To illustrate this, we design two sets of experiments. In the first one, depicted in the left plot of Figure~\ref{fig:diging}, we generate a matrix $W_1$ from the topology shown in Figure~\ref{fig:intro-2}, exhibiting a mild value of skewness $\kappa_\pi$ but large spectral gap $(1-\beta_\pi)^{-1}$. In the second one, depicted in the right plot of Figure~\ref{fig:diging}, we pick a matrix $W_2$ with a significant $\kappa_\pi$ and a negligible value of $1-\beta_\pi$. It is worth noting that, as per our analysis in Sec.~\ref{proof of section 6}, when employing $R$ rounds of gossip, the learning rate is supposed to amplify by $R$ times in a feasible regime. Therefore, we universally set the learning rate of \pushdg to $0.01$, while setting the learning rate of \mgpushdg to $0.01R$. Our numerical results in Figure~\ref{fig:diging} and~\ref{fig:MG diging} corroborate Theorem~\ref{thm:MG-Push-DIGing convergence}, showing that \mgpushdg alleviates the dependence on both $1-\beta_\pi$ and $\kappa_\pi$, compared to \pushdg.

\section{Conclusion}
This paper is concerned with smooth non-convex stochastic decentralized optimization over general directed networks using column-stochastic mixing matrices. A novel metric, called the equilibrium skewness, is introduced to facilitate explicit and complete characterization of the impact of directed networks on the convergence of decentralized optimization. In particular, we prove a tight lower bound for general decentralized algorithms using column-stochastic matrices that explicitly manifest the influence of the equilibrium skewness and spectral gap. Finally, we propose \mgpushdg, which combines \pushdg and multi-round gossip, and show that it attains our lower bound
up to logorithmic factors and is thus near-optimal. Experiments verify our theoretical findings.

\bibliographystyle{unsrt}
\bibliography{ref}

\appendix

\setcounter{table}{0}   
\setcounter{figure}{0}
\renewcommand{\thetable}{A\arabic{table}}
\renewcommand{\thefigure}{A\arabic{figure}}
\section{Proof of Lemma~\ref{push-sum lemma}}\label{appendix-Proof of section4}

\begin{proof}
 We prove the first statement. Let $a^{(k)}\triangleq \min_i \{{v_i^{(k)}}/{\pi_i}\}$ for any $k\geq 0$.
 Notice that 
 \vspace{-2pt}
	\begin{align}
		a^{(k+1)}
		= \min_i\left\{{\pi_i}^{-1}\textstyle\sum_{j=1}^n w_{ij}\pi_j {v_j^{(k)}}/{\pi_j} \right\} 
		\overset{(a)}{\geq} {\pi_i}^{-1}\textstyle\sum_{j=1}^n w_{ij}\pi_ja^{(k)}
		\overset{(b)}{=}a^{(k)},\nonumber
		\end{align}
where (a) holds because  $v_j^{(k)} > 0$ and $\pi_j > 0$ for any $1\leq j \leq n$ and $k\ge 0$, and (b) holds because of  Assumption~\ref{ass-weight-matrix}. The above inequality implies that $a^{(k)}$ is a non-decreasing sequence over $k$ and hence 
$a^{(k)} \ge a^{(0)}= \overline{\pi}^{-1}$. 

The second statement follows the same argument and is omitted here. To prove the third statement,  note that 
$\min_{i}\{v_i^{(k)}\} = \min_{i}\left\{{v_i^{(k)}\pi_i}/{\pi_i} \right\} \ge a^{(k)}\underline{\pi} \ge \overline{\pi}^{-1}\underline{\pi} = \kappa_\pi^{-1}, \forall k \ge 0$. Thus, we derive that $\|{V^{(k)}}^{-1}\|_2 = 1/{\min_i\{v_i^{(k)}\}} \le \kappa_\pi$.

For the fourth statement, with the fact that $\bar{x}^{(k)}=\bar{x}^{(0)}=n^{-1}\mathds{1}_n^\top x^{(0)}$, we have
\begin{align}
    \textstyle\|\vw^{(k)}-\bar{\vz}^{(k)}\|_F^2 &= \textstyle\sum_{j=1}^d\left\|{V^{(k)}}^{-1}(W^k-n^{-1}v^{(k)}\mathds{1}_n^\top)\vz_{\cdot,j}\right\|_2^2\nonumber\\
    &\textstyle=\textstyle\sum_{j=1}^d \left\|\mathrm{diag}(\pi)^{\frac{1}{2}}{V^{(k)}}^{-1}(W^k-n^{-1}v^{(k)}\mathds{1}_n^\top)\vz_{\cdot,j}\right\|_\pi^2 \nonumber\\
    &\textstyle\le \left\|\mathrm{diag}(\pi)^{\frac{1}{2}}{V^{(k)}}^{-1}\right\|_\pi^2 \textstyle\sum_{j=1}^d\left\|(W^k-n^{-1}v^{(k)}\mathds{1}_n^\top)\vz_{\cdot,j}\right\|_\pi^2. \nonumber
\end{align}
To bound the above term, we notice that
\begin{align*}
    \min_i\left\{{v_i^{(k)}}/{\sqrt{\pi_i}}\right\}\ge \min_i{\sqrt{v_i^{(k)}}}\min_i\left\{\sqrt{{v_i^{(k)}}/{\pi_i}}\right\} \ge {\sqrt{\kappa_\pi}^{-1}}\sqrt{\min_i\left\{{\pi_i}^{-1}\right\}}={\sqrt{\kappa_\pi\overline{\pi}}}^{-1}.
\end{align*}
Therefore, the multiplier $\left\|\mathrm{diag}(\pi)^{\frac{1}{2}}{V^{(k)}}^{-1}\right\|_\pi^2$ is bounded by $\left\|\mathrm{diag}(\pi)^{\frac{1}{2}}{V^{(k)}}^{-1}\right\|_\pi^2 =(1/\min_i\left\{v_i^{(k)}/\sqrt{\pi_i}\right\})^2\le \kappa_\pi \overline{\pi} $.
Next, the summation part is bounded by
\vspace{-2pt}
\begin{align}\label{eq:asd}
    &\textstyle\textstyle\sum_{j=1}^d\left\|\left(W^k-n^{-1}v^{(k)}\mathds{1}_n^\top\right)\vz_{\cdot,j}\right\|_\pi^2\\
    &\textstyle{\le}\textstyle\sum_{j=1}^d\|(W-\pi\mathds{1}_n^\top)^k \vz_{\cdot,j}\|_\pi^2+\sum_{j=1}^d\left\|\left(\pi\mathds{1}_n^\top-n^{-1}v^{(k)}\mathds{1}_n^\top\right)\vz_{\cdot,j}\right\|_\pi^2 \nonumber\\
    &\textstyle\overset{(e)}{\le}  \underline{\pi}^{-1}\beta_\pi^{2k}\left(1+\left\|\pi\mathds{1}_n^\top-v^{(0)}\mathds{1}_n^\top/n\right\|_\pi^2\right)\|\vz\|_F^2,\nonumber
\end{align}
where (e) holds because
\vspace{-2pt}
$$
   \textstyle \left\|\pi\mathds{1}_n^\top\hspace{-2pt}-\hspace{-2pt}v^{(k)}\mathds{1}_n^\top/n\right\|_\pi=\left\|(W\hspace{-2pt}-\hspace{-2pt}\pi\mathds{1}_n^\top)^k\left(\pi\mathds{1}_n^\top\hspace{-2pt}-\hspace{-2pt}v^{(0)}\mathds{1}_n^\top/n\right)\right\|_\pi\le \beta_\pi^k\left\|\pi\mathds{1}_n^\top\hspace{-2pt}-\hspace{-2pt}v^{(0)}\mathds{1}_n^\top/n\right\|_\pi\hspace{-1pt}.
$$
Further, note that 
\begin{align}\label{eq:asdd}
    \textstyle\left\|\pi\mathds{1}_n^\top-n^{-1}v^{(0)}\mathds{1}_n^\top\right\|_\pi& \textstyle=\left\|\sqrt{\pi}\sqrt{\pi}^\top-n^{-1}\Pi^{-1/2}v^{(0)}\sqrt{\pi}^\top\right\|_2 \\
   \textstyle \le \sqrt{\textstyle\sum_{i=1}^n{{(v^{(0)}_i)}^2}/{(n^2\pi_i)}-1} \le &\sqrt{\max_i\left\{{{v^{(0)}_i}/{(n\pi_i)}}\right\}\textstyle\sum_{i=0}^n{v^{(0)}_i}/{n}-1}= \sqrt{(n\underline{\pi})^{-1}-1}.\nonumber
\end{align}
The proof is completed by taking \eqref{eq:asdd} into \eqref{eq:asd} and bounding $(n\pi)^{-1}$ with $\kappa_\pi$.
\end{proof}

\section{Convergence of \pushdg}\label{appendix:proof of section 5}
We define the following notations for simplicity.
We define $\vg^{(k)}=\nabla F({V^{(k)}}^{-1}\vx^{(k)};\bxi^{(k+1)})\hspace{-2pt}\in \mathbb{R}^{n\times p}$. The bar symbol above letters indicates the averaging operation, \ie, $\bar{x}^{k}=n^{-1}\mathds{1}_n^\top\vx^{(k)}$, $\bar{y}^{(k)}=n^{-1}\mathds{1}_n^\top\vy^{(k)}$, $\bar{g}^{(k)}=n^{-1}\mathds{1}_n^\top \vg^{(k)}$, $\overline{\nabla f}^{(k)}=n^{-1}\mathds{1}_n^\top \nabla f(\vw^{(k)})$. We also define the averaging matrix $R=n^{-1}\mathds{1}_n\mathds{1}_n^\top$ and the weighted averaging matrix $J^{(k)}\hspace{-1pt}=\hspace{-1pt}n^{-1}v^{(k)}\mathds{1}_n^\top$. 
Besides, we define {\it the consensus error for $\vy$}: $\Delta_y^{(k)}\hspace{-1pt}:=\hspace{-1pt}\vy^{(k)}\hspace{-1pt}-\hspace{-1pt}J^{(k)}\vy^{(k)}$, {\it the consensus error for $\vx$}: $\Delta_x^{(k)}\hspace{-1pt}:=\hspace{-1pt}\vw^{(k)}\hspace{-1pt}-\hspace{-1pt}\bar{\vx}^{(k)}$ and {\it the increment for $\vg$}: $\Delta_g^{(k)}\hspace{-1pt}:=\hspace{-1pt}\vg^{(k+1)}\hspace{-1pt}-\hspace{-1pt}\vg^{(k)}$.

To proceed, we build up relationships between  $\Delta_x,\Delta_y,\Delta_g$ in Lemmas~\ref{lemma 7.1}--\ref{lemma GT consensus}. 

\begin{lemma} \label{lemma 7.1}
Under Assumptions~\ref{ass-weight-matrix} and~\ref{ass:sto}, the following holds for all $k,l \ge 0$:
\begin{enumerate}
    \item $\bar{y}^{(k)}=\bar g^{(k)},\forall k \ge 0$;
    \item $J^{(k)}W=J^{(k)},WJ^{(k)}=J^{(k+1)},(W-J^{(k+1)})J^{(k)}=0,J^{(l)}(W-J^{(k)})=0$;
    \item $W\vx-J^{(k+1)}\vx=(W-J^{(k)})(\vx-J^{(k)}\vx), \forall \vx \in \mathbb{R}^{n \times p}$;
    \item $\Delta_x^{(k+1)}=({V^{(k+1)}}^{-1}WV^{(k)}-RV^{(k)})\Delta_x^{(k)}-\gamma({V^{(k+1)}}^{-1}W-R)\Delta_y^{(k)}$;
    \item $\Delta_y^{(k+1)}=(W-J^{(k+1)})\Delta_y^{(k)}+(W-J^{(k+1)})\Delta_g^{(k)}$;
    \item $\pinorm{W^k-J^{(l)}}\le \beta_\pi^k+\sqrt{\kappa_\pi-1}\beta_\pi^l$. When $k\le l$, $\pinorm{W^k-J^{(l)}}\le \sqrt{2\kappa_\pi}\beta_\pi^k$;
    \item $\norm{W^k-J^{(l)}}_2\le \sqrt{\kappa_\pi}(\beta_\pi^k+\sqrt{\kappa_\pi-1}\beta_\pi^l)$. When $k\le l$, $\norm{W^k-J^{(l)}}_2\le \sqrt{2}\kappa_\pi\beta_\pi^k$;
    \item $\norm{{V^{(k+1)}}^{-1}WV^{(k)}-RV^{(k)}}_2\le \sqrt{2}\kappa_\pi^3\beta_\pi, \norm{{V^{(k+1)}}^{-1}W-R}_2\le\sqrt{2}\kappa_\pi^2\beta_\pi$.
     
\end{enumerate}    
\end{lemma}

\begin{proof}
Statement 1 follows from the algorithmic iterates. Statements 2--5 are valid by definitions. Statement 6 follows from the proof of Lemma~\ref{push-sum lemma}. Statement 7 holds because $\forall \vz \in \mathbb{R}^{n\times d}$, we have
\begin{align*}
    \textstyle\fnorm{(W^k-J^{(l)})\vz}^2\le \overline{\pi}\pinorm{W^k-J^{(l)}}^2\sum_{j=1}^n\pinorm{\vz_{\cdot,j}}^2\le \kappa_\pi\pinorm{W^k-J^{(l)}}^2\fnorm{\vz}^2.
\end{align*}
Statement 8 can be derived as follows:
\vspace{-2pt}
\begin{align*}
    &\textstyle\norm{{V^{(k+1)}}^{-1}WV^{(k)}-RV^{(k)}}_2\le\norm{{V^{(k+1)}}^{-1}}_2\norm{(W-J^{(k+1)})}_2\norm{V^{(k)}}_2\le \sqrt{2}\kappa_\pi^3\beta_\pi,\\
     &\textstyle\norm{{V^{(k+1)}}^{-1}W-R}_2\le \norm{{V^{(k+1)}}^{-1}}_2\norm{W-J^{(k+1)}}_2\le \sqrt{2}\kappa_\pi^2\beta_\pi.
\end{align*} 
where we use the inequalities $\pinorm{{V^{(k+1)}}^{-1}}\le \kappa_\pi$ and  $\pinorm{V^{(k)}}\le \kappa_\pi$. 
\end{proof} 

\begin{lemma}\label{lemma 7.2}
    The gradient increments can be bounded by
    \begin{align}
        \textstyle\mathbb{E}[\fnorm{\Delta_g^{(k)}}^2]&\le 6n\sigma^2+18\gamma^2L^2\kappa_\pi^4\beta_\pi^2\mathbb{E}[\fnorm{\Delta_y^{(k)}}^2]\\
        &\textstyle\quad +18\gamma^2L^2\mathbb{E}[\fnorm{\bar{\vg}^{(k)}}^2]+(9L^2+36L^2\kappa_\pi^6\beta_\pi^2)\mathbb{E}[\fnorm{\Delta_x^{(k)}}^2].\nonumber
    \end{align}
\end{lemma}
\begin{proof}
      By Jensen's inequality, we have: 
      \vspace{-2pt}
        \begin{align}\label{lemma 7.2-1}
       & \textstyle\mathbb{E}[\fnorm{\Delta_g^{(k)}}^2]=\mathbb{E}[\fnorm{\vg^{(k+1)}-\vg^{(k)}}^2] \\
       \le &\textstyle3\mathbb{E}[\fnorm{\vg^{(k+1)}-\nabla f^{(k+1)}}^2]+3\mathbb{E}[\fnorm{\nabla f^{(k+1)}-\nabla f^{(k)}}^2]+3\mathbb{E}[\fnorm{\nabla f^{(k)}-\vg^{(k)}}^2]\nonumber  \\
        \le &\textstyle 6n\sigma^2+3\mathbb{E}[\fnorm{\nabla f^{(k+1)}-\nabla f^{(k)}}^2], \nonumber 
    \end{align}
    where the last inequality is due to Assumption~\ref{ass:sto}. Using  Assumption~\ref{ass:smooth}, we have
\begin{align}\label{lemma 7.2-2}
       \textstyle \|\nabla f^{(k+1)}-\nabla f^{(k)}\|_F^2\le L^2\fnorm{\vw^{(k+1)}-\vw^{(k)}}^2.
    \end{align} 
    We can further apply Jensen's inequality to $\fnorm{\vw^{(l+1)}-\vw^{(k)}}^2$ and obtain
    \vspace{-2pt}
    \begin{align} \label{lemma 7.2-3}
        \|\vw^{(k+1)}-\vw^{(k)}\|_F^2 \le&  3\|\vw^{(k+1)}-\bar{\vx}^{(k+1)}\|_F^2+3\|\bar{\vx}^{(k+1)}-\bar{\vx}^{(k)}\|_F^2+3\|\Delta_x^{(k)}\|_F^2 \\
        = & 3\|\Delta_x^{(k+1)}\|_F^2+3\gamma^2\|\bar{\vg}^{(k)}\|_F^2+3\|\Delta_x^{(k)}\|_F^2 \nonumber \\
        \le &12\kappa_\pi^4\beta_\pi^2\gamma^2\|\Delta_y^{(k)}\|_F^2+3\gamma^2\|\bar{\vg}^{(k)}\|_F^2+(3+12\kappa_\pi^6\beta_\pi^2)\|\Delta_x^{(k)}\|_F^2,\nonumber
    \end{align}
    where the last inequality follows the 4-th and 8-th statements of Lemma~\ref{lemma 7.1}. We complete the proof by substituting  \eqref{lemma 7.2-2} and \eqref{lemma 7.2-3} into \eqref{lemma 7.2-1}.
\end{proof}

\begin{lemma}\label{lemma 7.3}
The consensus error for $\vy$ can be decomposed as
\vspace{-2pt}
    \begin{align}\label{lemma 7.3-1}
        \textstyle\|\Delta_y^{(k+1)}\|_F^2  &= \textstyle\sum_{l=0}^{k} \|(W^{k+1-l}-J^{(k+1)})\Delta_g^{(l)}\|_F^2\\
        \vspace{-2pt}
        &\textstyle\quad +  2\textstyle\sum_{l=0}^{k}\langle(W^{k+1-l}-J^{(k+1)})\Delta_y^{(l)},(W^{k+1-l}-J^{(k+1)})\Delta_g^{(l)} \rangle. \nonumber
    \end{align}
\end{lemma}
\begin{proof}
Using the 5-th statement of Lemma~\ref{lemma 7.1}, we have 
    \begin{align}\label{lemma 7.3-2}
        \textstyle\|\Delta_y^{(k+1)}\|_F^2&=\|(W-J^{(k+1)})\Delta_y^{(k)}+(W-J^{(k+1)})\Delta_g^{(k)}\|_F^2 \\
        &\textstyle=  \|(W-J^{(k+1)})\Delta_y^{(k)}\|_F^2+\|(W-J^{(k+1)})\Delta_g^{(k)}\|_F^2\nonumber\\
        &\textstyle\quad +2\langle(W-J^{(k+1)})\Delta_y^{(k)},(W-J^{(k+1)})\Delta_g^{(k)}\rangle. \nonumber
    \end{align}
    The first term in the last line of \eqref{lemma 7.3-2} can be decomposed similarly:
    \begin{align}
    \textstyle \|(W-J^{(k+1)})\Delta_y^{(k)}\|_F^2  &=\|(W^2-J^{(k+1)})\Delta_y^{(k-1)}\|_F^2+\|(W^2-J^{(k+1)})\Delta_g^{(k-1)}\|_F^2 \\
        & \textstyle\quad +2\langle(W^2-J^{(k+1)})\Delta_y^{(k-1)},(W^2-J^{(k+1)})\Delta_g^{(k-1)}\rangle, \nonumber
    \end{align}
where we use the relation $(W-J^{(k+1)})(W-J^{(k)})=W^2-J^{(k+1)}$. Repeating such a procedure yields the Lemma.
\end{proof}

Next, we bound each term of \eqref{lemma 7.3-1} in Lemmas~\ref{lemma 7.4} and~\ref{lemma 7.5}.
\begin{lemma}\label{lemma 7.4}
It holds for all $k\geq l \ge 0$ that
    \begin{align}\label{eq-lemma 7.4}
        &\textstyle\mathbb{E}[\|(W^{k+1-l}-J^{(k+1)})\Delta_g^{(l)}\|_F^2] \le 2\kappa_\pi^2\beta_\pi^{2(k+1-l)} C_{g,\sigma}\sigma^2\\
        &\textstyle+2\kappa_\pi^2\beta_\pi^{2(k+1-l)}\left(C_{g,y}\gamma^2\mathbb{E}[\fnorm{\Delta_y^{(l)}}^2]+C_{g,\bar{g}}\gamma^2\mathbb{E}[\fnorm{\bar{\vg}^{(l)}}^2]+C_{g,x}\mathbb{E}[\fnorm{\Delta_x^{(l)}}^2] \right) ,\nonumber
    \end{align} 
    where 
    $C_{g,\sigma}\triangleq 6n$, $C_{g,y}\triangleq 36L^2\kappa_\pi^4\beta_\pi^2$, $C_{g,\bar{g}}\triangleq 9L^2\kappa_\pi$, $C_{g,x}\triangleq 9L^2\kappa_\pi+36L^2\kappa_\pi^6\beta_\pi^2$.
\end{lemma}

\begin{proof}
With Lemma~\ref{lemma 7.1}~(7), the following inequality holds for all $k\ge l\ge 0$ :
    \begin{align} 
        \textstyle \mathbb{E}\|(W^{k+1-l}-J^{(k+1)})\Delta_g^{(l)}\|_F^2 
        \le 2\kappa_\pi^2\beta_\pi^{2(k+1-l)} \mathbb{E}\fnorm{\Delta_g^{(l)}}^2.\nonumber
        \end{align}
    By using Lemma~\ref{lemma 7.2} to bound $\fnorm{\Delta_g^{(l)}}^2$, we finish the proof.
\end{proof}

\begin{lemma}\label{lemma 7.5}
It holds for all $k\geq l \ge 0$ that
    \begin{align}\label{lemma 7.5-1}
        &\textstyle\mathbb{E}[\langle(W^{k+1-l}-J^{(k+1)})\Delta_y^{(l)},(W^{k+1-l}-J^{(k+1)})\Delta_g^{(l)} \rangle]  \\
        \le &\textstyle2\sigma^2\kappa_\pi^2\beta_\pi^{2k+2-2l}+\beta_\pi^{2(k-l)}\left(\frac{1-\beta_\pi}{6}+6\gamma L\kappa_\pi^4\beta_\pi^3\right)\mathbb{E}[\fnorm{\Delta_y^{(l)}}^2]  \nonumber \\
    &\textstyle+\beta_\pi^{2(k-l)}\frac{48\gamma^2L^2\kappa_\pi^4\beta_\pi^4}{1-\beta_\pi}\mathbb{E}[\fnorm{\bar{\vg}^{(l)}}^2]+\beta_\pi^{2(k-l)}\frac{48L^2(1+\kappa_\pi^3\beta_\pi)^2\kappa_\pi^4\beta_\pi^4}{1-\beta_\pi}\mathbb{E}[\fnorm{\Delta_x^{(l)}}^2].\nonumber
    \end{align}
\end{lemma}

\begin{proof}
{Noting that} 
$
    \mathbb{E}[\Delta_g^{(l)}|\cF^{(l)}]=\mathbb{E}[(\nabla f^{(l+1)}-\nabla f^{(l)})+(\nabla f^{(l)}-\vg^{(l)})|\cF^{(l)}]
$, the term in the left hand {side} of \eqref{lemma 7.5-1} can be decomposed to two terms of inner product. The first term is the inner product of $(W^{k+1-l}-J^{(k+1)})\Delta_y^{(l)}$ and $(W^{k+1-l}-J^{(k+1)})\Delta_y^{(l)}$, which can be bounded by the Cauchy-Schwarz inequality as follows
\begin{align}\label{lemma 7.5-2}
    &\mathbb{E}[\langle(W^{k+1-l}-J^{(k+1)})\Delta_y^{(l)},(W^{k+1-l}-J^{(k+1)})(\nabla f^{(l+1)}-\nabla f^{(l)}) \rangle]  \\
    \le &2L\kappa_\pi^2\beta_\pi^{2(k+1-l)}   \mathbb{E}[\|\Delta_y^{(l)}\|_F]\mathbb{E}[\| \vw^{(l+1)}-\vw^{(l)}\|_F ],\nonumber 
\end{align}
where the inequality is due to the 6-th statement Lemma~\ref{lemma 7.1} and  Assumption~\ref{ass:sto}. By further using triangle inequality in \eqref{lemma 7.2-3} to bound $\mathbb{E}\fnorm{\vw^{(l+1)}-\vw^{(l)}}$ in \eqref{lemma 7.5-2}, we obtain 
\begin{align}\label{lemma 7.5-3}
        \mathbb{E}[\langle(W^{k+1-l}&-J^{(k+1)})\Delta_y^{(l)},(W^{k+1-l}-J^{(k+1)})(\nabla f^{(l+1)}-\nabla f^{(l)}) \rangle]  \\
    \le & 6\gamma L\kappa_\pi^4 \beta_\pi^{2k+3-2l}  \mathbb{E}[\|\Delta_y^{(l)}\|_F^2] \nonumber\\
    +4L\kappa_\pi^2\beta_\pi^{2(k+1-l)}&(\gamma \mathbb{E}[\|\Delta_y^{(l)}\|_F]\mathbb{E}[\|\bar{\vg}^{(l)}\|_F]+(1+\kappa_\pi^3\beta_\pi) \mathbb{E}[\|\Delta_y^{(l)}\|_F]\mathbb{E}[\|\Delta_x^{(l)}\|_F] ).\nonumber
\end{align}
By Young inequality, we can further bound \eqref{lemma 7.5-3} as 
\begin{align}\label{lemma 7.5-4}
    &\mathbb{E}[\langle(W^{k+1-l}-J^{(k+1)})\Delta_y^{(l)},(W^{k+1-l}-J^{(k+1)})(\nabla f^{(l+1)}-\nabla f^{(l)}) \rangle]  \\
    \le & (6\gamma L\kappa_\pi^4 \beta_\pi^{2k+3-2l}+2L\kappa_\pi^2\beta_\pi^{2k+2-2l}\eta_1+2L\kappa_\pi^2\beta_\pi^{2k+2-2l}\eta_2 ) \mathbb{E}[\|\Delta_y^{(l)}\|_F^2] \nonumber\\
    &+2\gamma^2L\kappa_\pi^2\beta_\pi^{2k+2-2l}\eta_1^{-1}\mathbb{E}[\fnorm{\bar{\vg}^{(l)}}^2]+2L\kappa_\pi^2\beta_\pi^{2k+2-2l}(1+\kappa_\pi^3\beta_\pi)^2\eta_2^{-1}\mathbb{E}[\fnorm{\Delta_y^{(l)}}^2],\nonumber
\end{align}
where $\eta_1, \eta_2>0$ are to be determined. Later we will set $\eta_1=\eta_2=\frac{1-\beta_\pi}{24L\kappa_\pi^2\beta_\pi^2}$.

For the second term, which is the inner product of $(W^{k+1-l}-J^{(k+1)})\Delta_y^{(l)}$ and $(W^{k+1-l}-J^{(k+1)})(\nabla f^{(l)}-\vg^{(l)})$, since $(W^{k+1-l}-J^{(k+1)})J^{(l)}=0$, we have
\begin{align}
    &\textstyle\mathbb{E}[\langle(W^{k+1-l}-J^{(k+1)})\Delta_y^{(l)},(W^{k+1-l}-J^{(k+1)})(\nabla f^{(l)}-\vg^{(l)}) \rangle]\nonumber\\
    =&\textstyle\mathbb{E}[\langle(W^{k+1-l}-J^{(k+1)})\vy^{(l)},(W^{k+1-l}-J^{(k+1)})(\nabla f^{(l)}-\vg^{(l)}) \rangle]\nonumber\\
    =&\textstyle\mathbb{E}[\langle(W^{k+2-l}-J^{(k+1)})(\vy^{(l-1)}+\vg^{(l)}-\vg^{(l-1)}),(W^{k+1-l}-J^{(k+1)})(\nabla f^{(l)}-\vg^{(l)}) \rangle].\nonumber
\end{align}
Since $\vy^{(l-1)},\vg^{(l-1)}$ and $\nabla f^{(l)}$ are $\cF^{(l-1)}$-measurable, $\mathbb{E}[\nabla f^{(l)}-\vg^{(l)}|\cF^{(l-1)}]=0$. {Therefore}, we can further obtain that
\begin{align}
    &\textstyle\mathbb{E}[\langle(W^{k+1-l}-J^{(k+1)})\Delta_y^{(l)},(W^{k+1-l}-J^{(k+1)})(\nabla f^{(l)}-\vg^{(l)}) \rangle_\pi]\nonumber\\
    =&\textstyle\mathbb{E}[\langle(W^{k+2-l}-J^{(k+1)})(\vg^{(l)}-\nabla f^{(l)}),(W^{k+1-l}-J^{(k+1)})(\nabla f^{(l)}-\vg^{(l)}) \rangle_\pi]\nonumber\\
    =& \textstyle\mathbb{E}\left[\mathrm{tr}\left((\vg^{(l)}-\nabla f^{(l)})^\top(J^{(k+1)}-W^{k+1-l})^\top(W^{k+2-l}-J^{(k+1)})(\vg^{(l)}-\nabla f^{(l)})\right)\right]\nonumber.
\end{align}
By Assumption~\ref{ass:sto}, {the above expression} reduces to
\begin{align}\label{lemma 7.5-7}
        & \textstyle\mathbb{E}[\langle(W^{k+1-l}-J^{(k+1)})\Delta_y^{(l)},(W^{k+1-l}-J^{(k+1)})(\nabla f^{(l)}-\vg^{(l)}) \rangle ] \\
        =& \textstyle\mathbb{E}\left[\mathrm{tr}\left((\vg^{(l)}-\nabla f^{(l)})^\top\mathrm{diag}((J^{(k+1)}-W^{k+1-l})^\top(W^{k+2-l}-J^{(k+1)}))(\vg^{(l)}-\nabla f^{(l)})\right)\right]\nonumber\\
       \le &\textstyle \sigma^2\sum_{i=1}^n\lvert\sum_{j=1}^n(J^{(k+1)}-W^{k+1-l})_{ji}(W^{k+2-l}-J^{(k+1)})_{ji}\rvert\nonumber\\
       \le&\textstyle\sigma^2\sum_{i=1}^n\sqrt{\sum_{j=1}^n(J^{(k+1)}-W^{k+1-l})_{ji}^2\sum_{j=1}^n(W^{k+2-l}-J^{(k+1)})_{ji}^2}\nonumber\\
    \le&\sigma^2\fnorm{J^{(k+1)}-W^{k+1-l}}\fnorm{W^{k+2-l}-J^{(k+1)}} \le 2\sigma^2\kappa_\pi^2\beta_\pi^{2k+2-2l}\nonumber.
\end{align}
 Finally, combining \eqref{lemma 7.5-7} and \eqref{lemma 7.5-4}, we obtain \eqref{lemma 7.5-1}.
\end{proof}

\begin{lemma}\label{lemma 7.6}
     If $\gamma \le \frac{1-\beta_\pi}{40L\kappa_\pi^4\beta_\pi}$,  the accumulated consensus error of $\vy$ satisfies
     \vspace{-2pt}
    \begin{align}\label{lemma 7.6-1}
    \textstyle \sum_{k=0}^{T}\mathbb{E}[\fnorm{\Delta_y^{(k+1)}}^2]  
    \le& \textstyle \frac{16(T+1)n\kappa_\pi^2\beta_\pi^2\sigma^2}{1-\beta_\pi}
        +\frac{300L^2\kappa_\pi^{10}\beta_\pi^4}{(1-\beta_\pi)^2}\textstyle \sum_{l=0}^{T}\mathbb{E}[\fnorm{\Delta_x^{(l)}}^2]\\
        &\textstyle+\frac{150\gamma^2L^2\kappa_\pi^4\beta_\pi^2}{(1-\beta_\pi)^2} \sum_{l=0}^{T}\mathbb{E}[\fnorm{\bar{\vg}^{(l)} }^2] .\nonumber
    \end{align}
\end{lemma}

\begin{proof}
{Substituting} Lemmas~\ref{lemma 7.5} and~\ref{lemma 7.4} into Lemma~\ref{lemma 7.3}, we can verify that when $\gamma\le \frac{1-\beta_\pi}{20L\kappa_\pi^2\beta_\pi^2}$,
\vspace{-2pt}
\begin{align}
    \mathbb{E}[\fnorm{\Delta_y^{(k+1)}}^2]\le \textstyle\sum_{l=0}^k \beta^{2k-2l}\mathbb{E}\left[C_{y,\sigma}\sigma^2\hspace{-1pt}+\hspace{-1pt}C_{y,y}\fnorm{\Delta_y^{(l)}}^2\hspace{-1pt}+\hspace{-1pt}C_{y,x}\fnorm{\Delta_x^{(l)}}^2\hspace{-1pt}+\hspace{-1pt}C_{y,\bar{g}}\fnorm{\bar{\vg}^{(l)}}^2\right]\hspace{-1pt},\nonumber
\end{align}
where  $C_{y,\sigma}=16n\kappa_\pi^2\beta_\pi^2$, $C_{y,x}=\frac{96L^2(1+\kappa_\pi^3\beta_\pi)^2\kappa_\pi^4\beta_\pi^4}{1-\beta_\pi}+18L^2\kappa_\pi^4\beta_\pi^2+72L^2\kappa_\pi^8\beta_\pi^4$, $C_{y,y}=\frac{1-\beta_\pi}{2}$, $C_{y,\bar{g}}=18\gamma^2L^2\kappa_\pi^4\beta_\pi^2+\frac{96\gamma^2L^2\kappa_\pi^4\beta_\pi^4}{1-\beta_\pi}$. Summing it up 
from $k=0$ to $T$, we obtain that
\vspace{-2pt}
\begin{align}\label{lemma 7.6-3}
    &\textstyle\sum_{k=0}^T\mathbb{E}[\fnorm{\Delta_y^{(k+1)}}^2]\\
\le\textstyle\sum_{k=0}^T&\sum_{l=0}^k \beta^{2k-2l}\mathbb{E}\left[C_{y,\sigma}\sigma^2+C_{y,y}\fnorm{\Delta_y^{(l)}}^2+C_{y,x}\fnorm{\Delta_x^{(l)}}^2+\gamma^2C_{y,\bar{g}}\fnorm{\bar{\vg}^{(l)}}^2\right]\nonumber\\
\le \textstyle\sum_{l=0}^T&\mathbb{E}\left[C_{y,\sigma}\sigma^2+C_{y,y}\fnorm{\Delta_y^{(l)}}^2+C_{y,x}\fnorm{\Delta_x^{(l)}}^2+\gamma^2C_{y,\bar{g}}\fnorm{\bar{\vg}^{(l)}}^2\right]\textstyle\sum_{k=i}^T\beta^{2k-2l}\nonumber\\
\le \frac{1}{1-\beta^2}&\textstyle\sum_{l=0}^T\mathbb{E}\left[C_{y,\sigma}\sigma^2+C_{y,y}\fnorm{\Delta_y^{(l)}}^2+C_{y,x}\fnorm{\Delta_x^{(l)}}^2+\gamma^2C_{y,\bar{g}}\fnorm{\bar{\vg}^{(l)}}^2\right]\nonumber.
\end{align}
Using $\frac{C_{y,y}}{1-\beta^2}\le \frac{1}{2}$ in \eqref{lemma 7.6-3} gives \eqref{lemma 7.6-1}.
\end{proof}

\begin{lemma}\label{lemma GT consensus}
 Under Assumptions~\ref{ass-weight-matrix}--\ref{ass:smooth}, if $\gamma \le \frac{(1-\beta_\pi)^2}{40L\kappa_\pi^2\beta_\pi^3}$, we have 
 \vspace{-2pt}
    \begin{align}
   \textstyle\sum_{k=0}^{T}\mathbb{E}[\fnorm{\Delta_x^{(k+1)}}^2 ]&\le \left(\frac{64\gamma^2(T+1)n\kappa_\pi^6\beta_\pi^4}{(1-\beta_\pi)^3}+\frac{600\gamma^4(T+1)L^2\kappa_\pi^8\beta_\pi^4}{n(1-\beta_\pi)^4}\right)\sigma^2 \\
        &+\frac{600\gamma^4L^2\kappa_\pi^8\beta_\pi^4}{(1-\beta_\pi)^4}\textstyle\sum_{l=0}^{T}\mathbb{E}[\|\overline{\nabla f}^{(l)} \|_F^2]+\frac{4\gamma^2\kappa_\pi^4\beta_\pi^2}{(1-\beta_\pi)^2}\mathbb{E}[\fnorm{\Delta_y^{(0)}}^2]\nonumber
    \end{align}
\end{lemma}
\begin{proof}
    Iterating the 4-th statement of Lemma~\ref{lemma 7.1} from $k$ to 
$1$,  we have
    \vspace{-2pt}
    \begin{align}\label{lemma 7.7-2}
        \Delta_x^{(k+1)} =-\gamma \textstyle\sum_{l=0}^k({V^{(k+1)}}^{-1}W^{k+1-l}-R)\Delta_y^{(l)}
    \end{align}
    Thus, applying {Jensen's} inequality in \eqref{lemma 7.7-2}, we have
    \begin{align}
        \|\Delta_x^{(k+1)}\|_F^2  
    \le&\gamma^2\textstyle\sum_{l=0}^k\frac{\|{V^{(k+1)}}^{-1}W^{k+1-l}-R\|_2^2}{\beta_\pi^{k-l}(1-\beta_\pi)}\|\Delta_y^{(l)}\|_F^2 
    \le \frac{2\gamma^2\kappa_\pi^4\beta_\pi^2}{1-\beta_\pi}\textstyle\sum_{l=0}^k\beta_\pi^{k-l}\|\Delta_y^{(l)}\|_F^2.\nonumber
    \end{align}
    Summing it up from $k=0$ to $T$ similarly as in \eqref{lemma 7.6-3}, we obtain
    \vspace{-2pt}
    \begin{align}\label{lemma 7.7-4}
    \textstyle\sum_{k=0}^{T}\fnorm{\Delta_x^{(k+1)}}^2 \le \frac{2\gamma^2\kappa_\pi^4\beta_\pi^2}{(1-\beta_\pi)^2}\textstyle\sum_{l=0}^T\fnorm{\Delta_y^{(l)}}^2. 
    \end{align}
    Taking expectation on both sides of \eqref{lemma 7.7-4} and applying Lemma~\ref{lemma 7.6}, we obtain  
    \vspace{-2pt}
    \begin{align}\label{lemma 7.7-5}
    &\textstyle\sum_{k=0}^{T}\mathbb{E}[\fnorm{\Delta_x^{(k+1)}}^2]\le \frac{2\gamma^2\kappa_\pi^4\beta_\pi^2}{(1-\beta_\pi)^2}\mathbb{E}[\fnorm{\Delta_y^{(0)}}^2]+
    \frac{32\gamma^2(T+1)n\kappa_\pi^6\beta_\pi^4}{(1-\beta_\pi)^3}\sigma^2
        \\
        &\textstyle\quad +\frac{600\gamma^2L^2\kappa_\pi^{14}\beta_\pi^6}{(1-\beta_\pi)^4}\textstyle\sum_{l=0}^{T-1}\mathbb{E}[\fnorm{\Delta_x^{(l)}}^2]
        +\frac{300\gamma^4L^2\kappa_\pi^8\beta_\pi^4}{(1-\beta_\pi)^4}\sum_{l=0}^{T-1}\mathbb{E}[\fnorm{\bar{\vg}^{(l)} }^2].\nonumber
    \end{align}
    Since $\Delta_x^{(0)}=0$ and $\frac{600\gamma^2L^2\kappa_\pi^{14}\beta_\pi^6}{(1-\beta_\pi)^4}\le\frac{1}{2}$ when $\gamma \le \frac{(1-\beta_\pi)^2}{40L\kappa_\pi^7\beta_\pi}$, we can subtract \\$ \frac{1}{2}\sum_{l=0}^T\mathbb{E}[\fnorm{\Delta_x^{(l)}}^2] $ from both sides of \eqref{lemma 7.7-5}. Finally, by Assumption~\ref{ass:sto}, we have
    \vspace{-2pt}
    \begin{align*}
            \mathbb{E}[\|\bar{\vg}^{(k)}\|_F^2]
        =   \mathbb{E}[\|\overline{\nabla f}^{(k)}\|_F^2]+\mathbb{E}[\|\bar{\vg}^{(k)}-\overline{\nabla f}^{(k)}\|_F^2]
        \le \mathbb{E}[\|\overline{\nabla f}^{(k)}\|_F^2]+n^{-1}\sigma^2.
    \end{align*}
Thus, we can replace $\mathbb{E}[\|\bar{\vg}^{(k)}\|_F^2]$ in \eqref{lemma 7.7-5} and finish the proof of this lemma.
\end{proof}

\begin{lemma}\label{lemma:GT-descent}
    Under Assumptions~\ref{ass-weight-matrix} and~\ref{ass:sto}, if $\gamma < \frac{1}{2L}$, it holds for $\forall k \ge 0$ that
\vspace{-3pt}
\begin{align}
\mathbb{E}[f(\bar{x}^{(k+1)})|\cF_k] \le f(\bar{x}^{(k)})-\frac{\gamma}{2}\|\nabla f(\bar{x}^{(k)})\|^2-\frac{\gamma}{4}\|\overline{\nabla f}^{(k)}\|^2+\frac{\gamma L^2}{2n}\|\Delta_x^{(k)}\|_F^2+\frac{\gamma^2 L \sigma^2}{2n}. \nonumber
\end{align}
\end{lemma}

\begin{proof}
    Since $f$ is $L$-smooth, we have 
    \vspace{-2pt}
    \begin{align} \label{L-smooth}
         f(y)\le f(x)+ \langle \nabla f(x),y-x \rangle+L\|y-x\|^2/2.
    \end{align}
    With the second statement of Lemma~\ref{lemma 7.1}, we have $
    \bar{x}^{(k+1)}=\bar{x}^{(k)}-\gamma \bar{y}^{(k)}=\bar{x}^{(k)}-\gamma \bar{g}^{(k)}$. Thus, setting $y=\bar{x}^{(k+1)}$ and $x=\bar{x}^{(k)}$ in \eqref{L-smooth}, we obtain
    \vspace{-3pt}
    $$
    f(\bar{x}^{(k+1)})\le f(\bar{x}^{(k)})-\gamma \langle \nabla f(\bar{x}^{(k)}), \bar{g}^{(k)} \rangle+{\gamma^2L}\| \bar{g}^{(k)}\|^2/2.
    $$
    Conditioning on $\cF_k$ , since $\mathbb{E}[\bar{g}^{(k)}|\cF_k]=\overline{\nabla f}^{(k)}$, we have
    \vspace{-3pt}
    \begin{align*}
            & \textstyle\mathbb{E}[f(\bar{x}^{(k+1)})|\cF_k]-f(\bar{x}^{(k)}) \le -\gamma \langle\nabla f(\bar{x}^{(k)}),\overline{\nabla f}^{(k)}\rangle+{\gamma^2L}\mathbb{E}[\| \bar{g}^{(k)}\|^2|\cF_k] /2\nonumber \\
       \textstyle =   &\textstyle -\frac{\gamma}{2}\|\nabla f(\bar{x}^{(k)})\|^2-\frac{\gamma}{2}\|\overline{\nabla f}^{(k)}\|^2+\frac{\gamma}{2}\|\nabla f(\bar{x}^{(k)})-\overline{\nabla f}^{(k)}\|^2+\frac{\gamma^2L}{2}\mathbb{E}[\|\bar{g}^{(k)}\|^2|\cF_k] \nonumber \\
        \textstyle\le &\textstyle -\frac{\gamma}{2}\|\nabla f(\bar{x}^{(k)})\|^2-\frac{\gamma}{2}\|\overline{\nabla f}^{(k)}\|^2+\frac{\gamma L^2}{2n}\|\Delta_x^{(k)}\|_F^2+\frac{\gamma^2L}{2}\left(\|\overline{\nabla f}^{(k)}\|^2+\frac{ \sigma^2}{n}\right) \nonumber \\
        \le & \textstyle-\frac{\gamma}{2}\|\nabla f(\bar{x}^{(k)})\|^2-\frac{\gamma}{4}\|\overline{\nabla f}^{(k)}\|^2+\frac{\gamma L^2}{2n}\|\Delta_x^{(k)}\|_F^2+\frac{\gamma^2 L \sigma^2}{2n},
    \end{align*}
where the last inequality uses $\gamma \le \frac{1}{2L}$. 
\end{proof}

We present a detailed version of Theorem~\ref{thm:push diging convergence} in Theorem~\ref{thm:main theorem(appendix)}.
\begin{theorem}\label{thm:main theorem(appendix)}
Under Assumptions~\ref{ass-weight-matrix}--\ref{ass:sto}, by setting 
$\gamma=1/\sum_{i=1}^6\gamma_i^{-1}$ with 
$\gamma_1\triangleq\left(\frac{2n\Delta}{(K+1)L\sigma^2}\right)^{{1}/{2}}$, 
$\gamma_2\triangleq\left(\frac{\Delta (1-\beta_\pi)^3}{2(K+1)L^2\kappa_\pi^6\beta_\pi^4}\right)^{{1}/{3}}$, 
    $\gamma_3\triangleq\left(\frac{n^2\Delta(1-\beta_\pi)^4}{1200(K+1)L^4\kappa_\pi^8\beta_\pi^4\sigma^2}\right)^{{1}/{5}}$, 
    $\gamma_4\triangleq \left(\frac{(1-\beta_\pi)^2\Delta}{4L^2\kappa_\pi^4\beta_\pi^2\mathbb{E}[\fnorm{\vy^{(0)}}^2]}\right)^{{1}/{3}}$, 
    $\gamma_5\hspace{-2pt}\triangleq\hspace{-2pt}\frac{(1-\beta_\pi)^2}{40L\kappa_\pi^7\beta_\pi}$,
    $\gamma_6\hspace{-2pt}\triangleq\hspace{-2pt}\frac{1}{2L}$, and $\Delta\hspace{-2pt}\triangleq \hspace{-2pt}f(\bar{x}^{(0)})\hspace{-2pt}-\hspace{-2pt}\min_x \hspace{-1pt}f(x)$, we have
   \begin{align} \label{main theorem eq-1}
       &\textstyle\frac{1}{K+1}\textstyle\sum_{k=0}^{K} \mathbb{E}[\|\nabla f(\bar{x}^{(k)})\|^2 ]
        \le 2\left(\frac{2L\Delta\sigma^2}{n(K+1)}\right)^{{1}/{2}}+3\left(\frac{L^2\Delta^2\mathbb{E}[\fnorm{\vy^{(0)}}^2]\kappa_\pi^4\beta_\pi^2}{n(1-\beta_\pi)^2(K+1)^3}\right)^{{1}/{3}}\\
        &\textstyle\hspace{-2pt}+\hspace{-2pt}\left(\frac{12^3L^2\Delta^2\kappa_\pi^6\beta_\pi^4\sigma^2}{(1-\beta_\pi)^3(K+1)^2}\right)^{{1}/{3}}\hspace{-2pt}+\hspace{-2pt}\left(\frac{11^5L^4\Delta^4\kappa_\pi^5\beta_\pi^8\sigma^2}{n^2(1-\beta_\pi)^4(K+1)^4}\right)^{{1}/{5}} \hspace{-2pt}+\hspace{-2pt}\frac{80L\Delta\kappa_\pi^7\beta_\pi}{(1-\beta_\pi)^2(K+1)}\hspace{-2pt}+\hspace{-2pt}\frac{4L\Delta}{K+1}.\nonumber
    \end{align}
Consequently, the transient time is $K=\mathcal{O}\left(\frac{n^3\kappa_\pi^{14}}{(1-\beta_\pi)^6}\right)$.
\end{theorem}

\begin{proof}
    
    If $\gamma\le \frac{1}{2L}$, by summing up Lemma~\ref{lemma:GT-descent} from $k=0$ to $K$, we have
    \begin{align}\label{main theorem eq-2}
        & \textstyle\sum_{k=0}^{K} \mathbb{E}[\|\nabla f(\bar{x}^{(k)})\|^2] \\
       \textstyle \le  \frac{2}{\gamma}&\textstyle( f(\bar{x}^{(0)})-f^*)-\frac{1}{2}\textstyle\sum_{k=0}^{K}\mathbb{E}[\|\overline{\nabla f}^{(k)}\|^2]+\frac{L^2}{n}\textstyle\sum_{k=0}^{K}\mathbb{E}[\|\Delta_x^{(k)}\|_F^2]+\frac{(K+1)\gamma L\sigma^2}{n}.  \nonumber
    \end{align}
    Using Lemma~\ref{lemma GT consensus} in \eqref{main theorem eq-2}, we have
    \vspace{-2pt}
    \begin{align}\label{main theorem eq-3}
    &\textstyle \sum_{k=0}^{K} \mathbb{E}[\|\nabla f(\bar{x}^{(k)})\|^2] \nonumber\\
        \le &\textstyle \frac{2}{\gamma}( f(\bar{x}^{(0)})-f^*)-\left(\frac{1}{2}-\frac{600\gamma^4L^4\kappa_\pi^8\beta_\pi^4}{n(1-\beta_\pi)^4}\right) \textstyle\sum_{k=0}^{K-1}\mathbb{E}[\|\overline{\nabla f}^{(k)}\|^2]+\frac{(K+1)\gamma L\sigma^2}{n}  \\
            &\textstyle\quad +\frac{L^2}{n }\left(\frac{64\gamma^2(K+1)n\kappa_\pi^6\beta_\pi^4}{(1-\beta_\pi)^3}+\frac{600\gamma^4(K+1)L^2\kappa_\pi^8\beta_\pi^4}{n(1-\beta_\pi)^4}\right)\sigma^2 +\frac{4\gamma^2L^2\kappa_\pi^4\beta_\pi^2}{n(1-\beta_\pi)^2}\mathbb{E}[\fnorm{\Delta_y^{(0)}}^2]. \nonumber 
    \end{align} 
     In Lemma~\ref{lemma GT consensus} we need $\gamma\le \frac{(1-\beta_\pi)^2}{40L\kappa_\pi^7\beta_\pi}=\gamma_5$ and it is easy to verify $\frac{600\gamma^4L^4\kappa_\pi^8\beta_\pi^4}{n(1-\beta_\pi)^4}\le \frac{1}{2}$ when $\gamma\le \gamma_5$. Further, using the fact that $\fnorm{\Delta_y^{(0)}}^2=\fnorm{(I-R)\vy^{(0)}}^2\le\fnorm{\vy^{(0)}}^2$, we have
     \vspace{-2pt}
     \begin{align}\label{main theorem eq-4}
         &\textstyle \sum_{k=0}^{K} \mathbb{E}[\|\nabla f(\bar{x}^{(k)})\|^2]\le \frac{2}{\gamma}( f(\bar{x}^{(0)})-f^*)+\frac{(K+1)\gamma L\sigma^2}{n} \\
         &\textstyle+\left(\frac{64\gamma^2L^2K\kappa_\pi^6\beta_\pi^4}{(1-\beta_\pi)^3}+\frac{600\gamma^4L^4K\kappa_\pi^8\beta_\pi^4}{n^2(1-\beta_\pi)^4}\right)\sigma^2 +\frac{4\gamma^2L^2\kappa_\pi^4\beta_\pi^2}{n(1-\beta_\pi)^2}\mathbb{E}[\fnorm{\vy^{(0)}}^2]. \nonumber
     \end{align}
     Choosing $\gamma=1/\sum_{i=1}^6\gamma_i^{-1}\le\min_{1\le i\le 6}\{\gamma_i\}$ in \eqref{main theorem eq-4}, we complete the proof. 
\end{proof}

\section{Convergence of \mgpushdg}\label{proof of section 6}
By setting $R=\lceil \frac{p}{1-\beta_\pi}\rceil$ with $p=(1+\sqrt{7\ln(\kappa_\pi)})^2+(1+\sqrt{2\ln(n)})^2$, \mgpushdg converges as
\vspace{-2pt}
\begin{align}\label{eq:mgpushdg convergence}
   \textstyle\frac{1}{K+1}\sum_{k=0}^{K}\mathbb{E}[\|\nabla f(\bar{x}^{(k)})\|^2]=\mathcal{O}\left(\left(\frac{L\Delta\sigma^2}{nT}\right)^{{1}/{2}}+\frac{\Delta Lp}{T(1-\beta_\pi)}\right),
\end{align}
which implies a transient time of $\cO\left(\frac{n(1+\ln(\kappa_\pi))^2}{(1-\beta_\pi)^2}\right)$.
\begin{proof}
The difference between \mgpushdg and original \pushdg is that  $(W,\sigma^2)$ is replaced with $(W^R,\sigma^2/R)$. Note that $\|W^R-\pi\mathds{1}_n^\top\|_\pi=\|(W-\pi\mathds{1}_n^\top)^R\|_\pi\le \|W-\pi\mathds{1}_n^\top\|_\pi^R = \beta_\pi^R$. {Therefore},  $\beta_\pi$ {appearing} in the analysis of Appendix~\ref{appendix:proof of section 5} are all replaced by $\hat{\beta}:=\beta_\pi^R$. 
{Furthermore, as we accumulate $R$ stochastic gradients at each model, the gradient variance enjoys a $R$-times smaller bound, \ie, $\hat{\sigma}^2=\sigma^2/R$.} With $R=\lceil\frac{p}{1-\beta_\pi}\rceil$ and Theorem~\ref{thm:main theorem(appendix)}, we directly obtain
\vspace{-2pt}
\begin{align}\label{eq:mgpushdg-2}
&\textstyle\frac{1}{K+1}\sum_{k=0}^{K}\mathbb{E}[\|\nabla f(\bar{x}^{(k)})\|^2]=\cO\left(\left(\frac{L\Delta\hat{\sigma}^2}{nK}\right)^{{1}/{2}}+\left(\frac{L^2\Delta^2\kappa_\pi^4\hat{\beta}^2}{nK^3}\right)^{{1}/{3}}\right)\\
&\textstyle\quad+\mathcal{O}\left(\left(\frac{L^2\Delta^2\kappa_\pi^6\hat{\beta}^4\hat{\sigma}^2}{K^2}\right)^{{1}/{3}}+\left(\frac{L^4\Delta^4\kappa_\pi^8\hat{\beta}^4\hat{\sigma}^2}{n^2K^4}\right)^{{1}/{5}}+\frac{L\Delta (1+\kappa_\pi^7\hat{\beta})}{K}\right) \nonumber \\
\textstyle\overset{T=(K+1)R}{=} &\textstyle\mathcal{O}\left(\left(\frac{L\Delta\sigma^2}{nT}\right)^{{1}/{2}}+\left(\frac{L^2\Delta^2p^3\kappa_\pi^4\hat{\beta}^2}{nT^3(1-\beta_\pi)^3}\right)^{{1}/{3}}+\left(\frac{L^2\Delta^2p\kappa_\pi^6\hat{\beta}^4\sigma^2}{T^2(1-\beta_\pi)}\right)^{{1}/{3}}\right)\nonumber\\
&\textstyle\quad+\mathcal{O}\left(\left(\frac{L^4\Delta^4p^3\kappa_\pi^8\hat{\beta}^4\sigma^2}{n^2T^4(1-\beta_\pi)^4}\right)^{{1}/{5}}+\frac{L\Delta p(1+\kappa_\pi^7\hat{\beta})}{T(1-\beta_\pi)} \right)\nonumber 
\end{align}
where $1-\hat{\beta}$, $\mathbb{E}[\fnorm{\Delta_y^{(0)}}^2]=\cO(1)$ are omitted. Notice that
$$p\hat{\beta} \le p {(\beta_\pi^{1/(1-\beta_\pi)})}^{p} \le p e^{-p}\le \frac{(1\sqrt{7\ln(\kappa_\pi)})^2+(1+\sqrt{2\ln(n)})^2}{e^{2+2\sqrt{7\ln(\kappa_\pi)}+2\ln(n)}n^2\kappa_\pi^7} \le n^{-2}\kappa_\pi^{-7}, $$ 
we have $p^3\kappa_\pi^4\hat{\beta}^2\le1$, $p\kappa_\pi^6\hat{\beta}^4\le n^{-1}$, $p^3\kappa_\pi^8\hat{\beta}^4\le1$ and $p\kappa_\pi^7\hat{\beta}\le1$. {Applying} them to \eqref{eq:mgpushdg-2}, we obtain
\vspace{-2pt}
\begin{align}\label{eq:mgpushdg-3}
    &\textstyle\frac{1}{K+1}\sum_{k=0}^{K}\mathbb{E}[\|\nabla f(\bar{x}^{(k)})\|^2]=\mathcal{O}\left(\left(\frac{L\Delta\sigma^2}{nT}\right)^{{1}/{2}}+\left(\frac{L^{2}\Delta^{2}}{nT^3(1-\beta_\pi)^3}\right)^{1/3}\right)\\
    &\textstyle\quad \textstyle+\mathcal{O}\left(\left(\frac{L^2\Delta^2\sigma^2}{nT^2(1-\beta_\pi)}\right)^{{1}/{3}} +\left(\frac{L^4\Delta^4\sigma^2}{n^2T^4(1-\beta_\pi)^4}\right)^{{1}/{5}}+\frac{L\Delta p}{T(1-\beta_\pi)}\right). \nonumber 
\end{align} 
Further by Young's inequality, we have $\left(\frac{L^2\Delta^2\sigma^2}{nT^2(1-\beta_\pi)}\right)^{{1}/{3}}\le \frac{2}{3}\left(\frac{L\Delta\sigma^2}{nT}\right)^{{1}/{2}}\hspace{-1pt}+\frac{1}{3}\frac{L\Delta }{T(1-\beta_\pi)}$ and $\left(\frac{L^4\Delta^4\sigma^2}{n^2T^4(1-\beta_\pi)^4}\right)^{{1}/{5}}\le \frac{1}{5}\left(\frac{L\Delta\sigma^2}{nT}\right)^{{1}/{2}}+\frac{4}{5}\frac{L\Delta }{T(1-\beta_\pi)}$. We thus obtain \eqref{eq:mgpushdg convergence}.
\end{proof}

\section{Proof of Proposition~\ref{prop:special-mat}}\label{app:special-mat}
For any $n\geq 2$, let 
\vspace{-2pt}
\[\textstyle
J:=\begin{bmatrix}
0 & 0 &\cdots& 0 & 1\\
0 & 0 & & &  \\
 &\ddots  &\ddots &   \\
 &  & 0 &  0&  \\
 &  & & 0& 0\\
\end{bmatrix}\quad \text{and}\quad W_{\epsilon}:=\frac{1+\epsilon}{2}J+\frac{1-\epsilon}{2}e_{1}\mathds{1}_n^\top,
\]
for $\epsilon\in(-1,1)$, where $e_{1}$ is the first canonical vector in $\RR^n$. One can verify that 
$W_{\epsilon}$'s right Perron vector is
\vspace{-2pt}
\[\textstyle
\pi\propto \left((2/(1+\epsilon))^{n-1},(2/(1+\epsilon))^{n-2},\dots, 2/(1+\epsilon),1\right)^\top\in\RR^n,
\]
which directly implies $\ln(\kappa_\pi)=(n-1)\ln(2/(1+\epsilon))$.
One can also easily verify that 
\vspace{-2pt}
\begin{equation*}\textstyle
    [\pi]_1=\frac{1-\epsilon}{2(1-((1+\epsilon)/2)^{n-1})}> \frac{1-\epsilon}{2}.
\end{equation*}
For any $v\in\RR^n$, denoting $\one_n^\top v$ as $s$,  we have
\vspace{-2pt}
\begin{align}
   & \textstyle\|W_{\epsilon}v\|_\pi^2 = (W_{\epsilon}v)^\top \Pi^{-1}(W_{\epsilon}v)   \nonumber\\
   =& \textstyle\left(\frac{1+\epsilon}{2}Jv+\left(\frac{1-\epsilon}{2}e_{1}-\pi\right)s \right)^\top \Pi^{-1}\left(\frac{1+\epsilon}{2}Jv+\left(\frac{1-\epsilon}{2}e_{1}-\pi\right)s\right) \nonumber\\
   =&\textstyle \frac{(1+\epsilon)^2}{4}(Jv)^\top \Pi^{-1}Jv+ (1+\epsilon)(Jv)^\top \Pi^{-1}\left(\frac{1-\epsilon}{2}e_{1}-\pi\right)s \nonumber\\
   &\textstyle\quad +\left(\frac{1-\epsilon}{2}e_{1}-\pi\right)^\top\Pi^{-1}\left(\frac{1-\epsilon}{2}e_{1}-\pi\right)s^2.\nonumber
\end{align}
Using $J^\top \Pi^{-1}J=\frac{2}{1+\epsilon}\Pi^{-1}+\left([\pi]_1^{-1}-\frac{2}{1+\epsilon}[\pi]_n^{-1}\right)e_{n}e_{n}^\top$, and $\Pi^{-1}e_{1}=e_{1}/[\pi]_1$, $\Pi^{-1}\pi=\one_n$, we further have
\vspace{-2pt}
\begin{align}
      \|W_{\epsilon}v\|_\pi^2 =&\textstyle\frac{(1+\epsilon)^2}{4}v^\top\left(\frac{2}{1+\epsilon}\Pi^{-1}+\left([\pi]_1^{-1}-\frac{2}{(1+\epsilon)[\pi]_n}\right)e_ne_n^\top)\right)v\nonumber\\
      &\textstyle + (1+\epsilon)\left(\frac{1-\epsilon}{2[\pi]_1}v_n-s\right)s +\left(\frac{(1-\epsilon)^2}{4[\pi]_1}+\epsilon\right)s^2\label{eqn:hvisnvxv2} \\
   =\textstyle\frac{1+\epsilon}{2}\|v\|_\pi^2 -&\textstyle\left(\left(\frac{1+\epsilon}{2[\pi]_n}-\frac{(1+\epsilon)^2}{4[\pi]_1}\right)v_n^2-\frac{1-\epsilon^2}{2[\pi]_1}v_n s +\left(1-\frac{(1-\epsilon)^2}{4[\pi]_1}\right)s^2\right)\nonumber.
\end{align}
Considering the determinant of the quadratic function with respect to $v_n$ and $s$ in  \eqref{eqn:hvisnvxv2}, we have
\vspace{-2pt}
\begin{align}
   &\textstyle\left(\frac{1-\epsilon^2}{2[\pi]_1}\right)^2-4\times \frac{(1+\epsilon)^2}{4}\left(\frac{2}{(1+\epsilon)[\pi]_n}-\frac{1}{[\pi]_1}\right)\times \left(1-\frac{(1-\epsilon)^2}{4[\pi]_1}\right)\nonumber\\
   =&\textstyle\frac{(1+\epsilon)^2}{4[\pi]_1^2}\left((1-\epsilon)^2-(1+\epsilon)^2\left(\frac{2[\pi]_1}{(1+\epsilon)[\pi]_n}-1\right)\left(4[\pi]_1-(1-\epsilon)^2\right)\right)\nonumber\\
   < &\textstyle\frac{(1+\epsilon)^2}{4[\pi]_1^2}\left((1-\epsilon)^2-(1+\epsilon)^2\frac{1-\epsilon}{1+\epsilon}(1-\epsilon^2)\right)=\frac{(1-(1+\epsilon)^2)(1-\epsilon)^2}{4[\pi]_1^2}\leq 0\label{eqn:vbucxbvcz}
\end{align}
where \eqref{eqn:vbucxbvcz} is due to $[\pi]_1\geq [\pi]_n$ and $[\pi]_1> (1-\epsilon)/2$. Since the determinant is non-positive, combining this with \eqref{eqn:hvisnvxv2}, we conclude $\|W_{\epsilon}v\|_\pi^2 \leq \frac{1+\epsilon}{2}\|v\|_\pi^2$ and the equality holds for $v$ with $v_n=0$ and $s=\sum_{i=1}^{n-1}v_i=0$. Therefore, we have $\beta_\pi =\sqrt{(1+\epsilon)/2}\in(0,1)$, $\ln(\kappa_\pi)=-2(n-1)\ln(\beta_\pi)$, and consequently
\begin{equation*}\textstyle
    \frac{1+\ln(\kappa_\pi)}{1-\beta_\pi}=\frac{1-2(n-1)\ln(1-(1-\beta_\pi))}{1-\beta_\pi}=O(n)
\end{equation*}
provided with $\Omega(1)= \beta_\pi\leq 1-1/n$ so that $-\ln(\beta_\pi)/(1-\beta_\pi)=O(1)$.

\section{Proof of Theorem~\ref{thm:lower-bound}}\label{app:lower-bound}
The first complexity $\Omega(\frac{\sigma\sqrt{L\Delta}}{\sqrt{nK}})$ is customary, whose proof can be found in \eg, \cite{lu2021optimal,yuan2022revisiting}. We thus focus on proving the second term $\Omega((1+\ln(\kappa_{\pi}))L\Delta/K)$. To proceed, we denote the $j$-th coordinate of a vector $x\in\RR^d$ by $[x]_j$ for $1\leq j\leq d$, and define
\vspace{-2pt}
\begin{equation*}\textstyle
    \prog(x):=\begin{cases}
        0 &\text{if }x=0;\\
        \max_{1\leq j\leq d}\{j:[x]_j\neq 0\}&\text{othwewise}.
    \end{cases}
\end{equation*}
We also present several key lemmas, which have appeared in prior literature.
\begin{lemma}[Lemma 2 of \cite{Arjevani2019LowerBF}]\label{lem:basic-fun}
Let function 
\vspace{-2pt}
\begin{equation*}\textstyle
    h(x):=-\psi(1) \phi([x]_{1})+\sum_{j=1}^{d-1}\Big(\psi(-[x]_j) \phi(-[x]_{j+1})-\psi([x]_j) \phi([x]_{j+1})\Big)
\end{equation*}
where for $\forall\, z \in \mathbb{R}$,
\vspace{-2pt}
$$
\psi(z)=\begin{cases}
0 & z \leq 1 / 2; \\
\exp \left(1-\frac{1}{(2 z-1)^{2}}\right) & z>1 / 2, 
\end{cases} \quad \quad \mbox{and} \quad \quad  \phi(z)=\sqrt{e} \int_{-\infty}^{z} e^{-\frac{1}{2} t^{2}} \mathrm{d}t.
$$
The function $h(x)$ satisfies the following properties:
\begin{enumerate}
    \item $h$ is zero-chain, \ie, $\prog(\nabla h(x))\leq \prog(x)+1$ for all $x\in\RR^d$.
    \item $h(x)-\inf_{x} h(x)\leq \Delta_0 d$, $\forall\,x\in\RR^d$ with $\Delta_0=12$.
    \item $h$ is $L_0$-smooth with $L_0=152$.
    \item $\|\nabla h(x)\|_\infty\leq G_\infty $, $\forall\,x\in\RR^d$ with $G_\infty = 23$.
    \item $\|\nabla h(x)\|_\infty\ge 1 $ for any $x\in\RR^d$ with $[x]_d=0$. 
\end{enumerate}
\end{lemma}

\begin{lemma}[Lemma 4 of \cite{huang2022lower}]\label{lem:basic-fun2}
Letting functions 
\vspace{-2pt}
\begin{equation*}\textstyle
    h_1(x):=-2\psi(1) \phi([x]_{1})+2\sum_{j \text{ even, } 0< j<d}\Big(\psi(-[x]_j) \phi(-[x]_{j+1})-\psi([x]_j) \phi([x]_{j+1})\Big)
\end{equation*}
and 
\vspace{-2pt}
\begin{equation*}\textstyle
    h_2(x):=2\sum_{j \text{ odd, } 0<j<d}\Big(\psi(-[x]_j) \phi(-[x]_{j+1})-\psi([x]_j) \phi([x]_{j+1})\Big),
\end{equation*}
then $h_1$ and $h_2$ satisfy  the following properties:
\begin{enumerate}
    \item $\frac{1}{2}(h_1+h_2)=h$, where $h$ is defined in Lemma \ref{lem:basic-fun}.
    \item $h_1$ and $h_2$ are zero-chain, \ie, $\prog(\nabla h_i(x))\leq \prog(x)+1$ for all $x\in\RR^d$ and $i=1,2$. Furthermore, if $\prog(x)$ is odd, then $\prog(\nabla h_1(x))\leq \prog(x)$; if $\prog(x)$ is even, then $\prog(\nabla h_2(x))\leq \prog(x)$.
    \item $h_1$ and $h_2$ are also $L_0$-smooth with ${L_0}=152$. 
\end{enumerate}
\end{lemma}

Our proof proceeds in three steps. Without loss of generality, we assume $n$ can be divided by $3$.

\vspace{2mm}
\noindent (Step 1.) We let $f_i=L\lambda^2 h_1(x/\lambda)/L_0$, $\forall\,i\in E_1\triangleq\{j:1\leq j\leq n/3\}$ and $f_i=L\lambda^2 h_2(x/\lambda)/L_0$, $\forall\,i\in E_2\triangleq\{j:2n/3 \leq j\leq n\}$, where $h_1$ and $h_2$ are defined in Lemma \ref{lem:basic-fun2}, and $\lambda>0$ will be specified later. By the definitions of $h_1$ and $h_2$, we have that $f_i$, $\forall\,1\leq i\leq n$, is  zero-chain  and $f(x)=n^{-1}\sum_{i=1}^n f_i(x)=2L\lambda^2 h(x/\lambda)/3L_0$. Since $h_1$ and $h_2$ are also $L_0$-smooth, $\{f_i\}_{i=1}^n$ are $L$-smooth. Furthermore, since
\vspace{-2pt}
\begin{equation*}\textstyle
    f(0)-\inf_x f(x)=\frac{2L\lambda^2}{3L_0}(h(0)-\inf_x h(x)) {\leq}\frac{L\lambda^2\Delta_0d}{L_0},
\end{equation*}
to ensure $\{f_i\}_{i=1}^n$ satisfy Assumption~\ref{ass:smooth}, it suffices to let
\vspace{-2pt}
\begin{equation}\label{eqn:jgowemw}\textstyle
    \frac{L\lambda^2\Delta_0d}{L_0}\leq  \Delta, \quad \text{\ie,}\quad \lambda\leq \sqrt{\frac{L_0 \Delta}{L\Delta_0 d}}.
\end{equation}
With the functions defined above, we have  $f(x)=n^{-1}\sum_{i=1}^n f_i(x)=L\lambda^2 l(x/\lambda)/(3L_0)$ and $\prog(\nabla f_i(x)) =\prog(x) +1$ if $\prog(x)$ is even and $i\in E_1$ or $\prog(x)$  is odd  and $i\in E_2$, otherwise $\prog(\nabla f_i(x))\leq \prog(x) $.
Therefore, to make progress (\ie, to increase $\prog(x)$), for any gossip algorithm $A\in\cA_{W}$, one must take the gossip communication protocol to transmit information between $E_1$ and $E_2$ alternatively.

\vspace{2mm}
\noindent (Step 2.) We consider the noiseless gradient oracles and the constructed mixing matrix $W$ in Appendix~\ref{app:special-mat} with $\epsilon=2\beta_\pi^2-1$ so that $\frac{1+\ln(\kappa_\pi)}{1-\beta_\pi}=O(n)$. Note the directed distance from $E_1$ to $E_2$ is $n/3$. Consequently,  starting from $x^{(0)}=0$, it takes of at least $n/3$ communications for any possible algorithm $A\in\cA_W$ to increase  $\prog(\hat{x})$  by $1$ if it is odd. Therefore, we have
$
    \left \lceil \prog(\hat x^{(k)})/2\right \rceil \leq \left\lfloor \frac{k}{2n/3}\right\rfloor,\forall\,k\geq 0.
$
This further implies
\vspace{-2pt}
\begin{equation}\label{eqn:jofqsfdq}\textstyle
    \prog(\hat x^{(k)})\leq 2\left\lfloor \frac{k}{2n/3}\right\rfloor+1\leq  3k/n+1,\quad \forall\,k\geq 0.
\end{equation}

\vspace{2mm}
\noindent (Step 3.)
We finally show the error $\EE[\|\nabla f(x)\|^2]$ is lower bounded by $\Omega\left(\frac{(1+\ln(\kappa_\pi))L\Delta }{(1-\beta_\pi)K}\right)$, with any algorithm $A\in\cA_{W}$ with $K$ communication rounds. 
For any $K\geq n$, we set
$
    d= 2\left\lfloor \frac{K}{2n/3}\right\rfloor+2 \leq 3K/n+2\leq 5K/n
$
and
$
     \lambda =\left(\frac{nL_0 \Delta}{5L \Delta_0 K}\right)^{{1}/{2}}.
$
Then \eqref{eqn:jgowemw} naturally holds.
Since $\prog(\hat x^{(K)})<d$ by \eqref{eqn:jofqsfdq}, using the last point of  Lemma~\ref{lem:basic-fun} and the value of $\lambda$, we obtain
\vspace{-2pt}
\begin{equation*}\textstyle
    \EE[\|\nabla f(\hat{x})\|^2]\geq\min_{[\hat{x}]_{d}=0}\|\nabla f(\hat{x})\|^2\geq \frac{L^2\lambda^2}{9L_0^2}=\Omega\left(\frac{nL\Delta }{K}\right).
\end{equation*}
By finally using $n=\Omega((1+\ln(\kappa_\pi))/(1-\beta_\pi))$, we complete the proof.


\end{document}